\documentclass[a4paper,11pt]{amsart}

\usepackage{hyperref}
\usepackage{comment}
\usepackage[english]{my-shortcuts}
\usepackage{srcltx,algorithm,algorithmic}
\usepackage{amsmath}
\usepackage{amsfonts,enumerate}

\usepackage{xcolor}

\renewcommand{\sgn}{\mathrm{sgn}}
\DeclareMathOperator*{\argmin}{\mathrm{Argmin}}

\begin{document}

\title[Mixture model for designs in high dimensional regression and the LASSO]
{Mixture model for designs in high dimensional regression and the LASSO}

\author{Mohamed Ibrahim Assoweh}
\address{Laboratoire de Math\'ematiques de Besan{\c c}on, Univ. de Bourgogne Franche-Comt\'e, 16 route de Gray, 25030, Besan{\c c}on, France}
\address{Universit\'e de Djibouti, Campus de Balbala, Croisement RN2-RN5, B.P: 1904, Djibouti}
\email{assowehmohamed@yahoo.fr}

\author{Emmanuel Caron}
\address{Laboratoire de Math\'ematiques d'Avignon, Univ. d'Avignon, France}

\email{emmanuel.caron-parte@ univ-avignon.fr}

\author{St\'ephane Chr\'etien}

\address{ERIC Laboratory and UFR ASSP,
University of Lyon 2, 
5 avenue Mendes France,
69676 Bron Cedex, France}

\email{stephane.chretien@univ-lyon2.fr}

\maketitle

%\vspace{.5cm}

\begin{abstract}
The LASSO is a recent technique for variable selection in the regression model 
\bean 
y & = & X\beta + z,
\eean 
where $X\in \R^{n\times p}$ and $z$ is a centered gaussian i.i.d. noise vector $\mathcal N(0,\sigma^2I)$. 
The LASSO has been proved to achieve remarkable properties such as exact support recovery of sparse vectors when the columns are sufficently incoherent and low prediction error
under even less stringent conditions. 
However, many matrices do not satisfy small coherence in practical applications and the LASSO estimator may thus suffer from what is known as the slow rate regime.

The goal of the present paper is to study the LASSO from a slightly different perspective by proposing a mixture model for the design matrix which is able to capture in a natural way the potentially clustered nature of the columns in many practical situations. In this model, the columns of the design matrix are drawn from a Gaussian mixture model. Instead of requiring incoherence for the design matrix $X$, we only require incoherence of the much smaller matrix of the mixture's centers. 

Our main result states that $X\beta$ can be estimated with the same precision as for incoherent designs except for a correction term depending on the maximal variance in the mixture model.  
\end{abstract}
%Lasso = probleme convexe mais pas strictement convexe. Pas de solution explicite -> algo de descente
%Pas de solution unique, mais selon certains resultats recents, sous une condition generique, on peut avoir une solution unique.

%\tableofcontents

\section{Introduction}
The goal of the present paper is the study of the high dimensional regression problem 
$y = X\beta + z$, where $X \in \R^{n \times p}$, with $p \gg n$ and 
$z \sim \mathcal{N}(0,\sigma^2 I)$. 
This problem has been the subject of an extensive research activity. This high dimensional setting, where more variables are 
involved than observations, occurs in many different applications such as 
image processing and denoising, gene expression analysis,  time series (filtering) \cite{KimEtAl:SIREV09}, \cite{NetoSardyTseng:JCGS12}, graphical models 
\cite{MeinshausenBuhlmann:AnnStat06}, biochemistry \cite{AlquraishiAdams:PNAS12}, etc. 
One very popular approach is the 
Least Angle Shrinkage and Selection Operator (LASSO) introduced in 
\cite{Tibshirani:JRSSB96} for the purpose of variable selection. The LASSO 
estimator is given as a solution\footnote{Conditions for uniqueness of the minimizer in this last expression are 
	discussed in \cite{Fuchs:IEEEIT02}, \cite{Osborne:IMANUMAN00},\cite{ChretienDarses:ArXiv11} and \cite{Dossal:Hal07}
	}, for $\lb>0$, of 
\bea
\label{lasso}
\hat \beta = \argmin_{b \in \R^p}\ \frac12 \|y - Xb\|_2^2 + \lambda\ \|b\|_1.
\eea
The main advantage of the LASSO over more traditional penalized likelihood optimization procedures such as BIC, AIC, etc, is that a solution can be obtained in polynomial time by solving a convex optimisation problem. Very efficient scalable algorithms are available, based on Nesterov's method \cite{Becker:SIAMImSc10}, the Alternating Direction of Method of Multipliers \cite{parikh:proximal}, the Frank-Wolfe algorithm \cite{jaggi:revisiting} or online versions of them \cite{LafondWaiMoulines}, \cite{ChretienGibbertRoyArXiv18}. 

One of the most surprising and important discoveries is that, under appropriate assumptions on the design 
matrix $X$, and for at least most regression vectors $\beta$, the support of $\beta$ can be 
recovered exactly when its size is up to the order of $n/\log(p)$ and the nonzero components are sufficiently large; see 
\cite{Bickel:AnnStat09}, \cite{Bunea:EJS07}, \cite{CandesPlan:AnnStat09}, 
\cite{Wainwright:IEEEIT09} for instance. Moreover, under 
similar assumptions, the prediction error can be controlled adaptively as 
a function of the sparsity of $\beta$ and the noise variance; see for instance
\cite{CandesPlan:AnnStat09}. 
%sufficiently large = suffisamment détectable (par rapport au bruit)

A great amount of work has been devoted to finding 
error bounds on $X\beta$ \cite{CandesPlan:AnnStat09}, \cite{ChandrasekaranRechtParriloWillsky}, \cite{Zhang:sharp}, \cite{ChretienDarses:IEEEIT12}, etc. Oracle inequalities for this problem are divided into two different classes depending on the so called "regime": the first class describes the slow rate regime and does not require any particular assumption on $X$, while the second class describes the fast rate regime which does require some structural assumptions on $X$.

\vspace{.5cm}

\begin{center}
    \textit{For simplicity, we will assume 
throughout this paper that the columns of $X$ have unit $\ell_2$-norm. }
\end{center}

\vspace{.5cm}

Historically, the first is the Restricted 
Isometry Property \cite{CandesTao:IEEEIT06} \cite{Candes:CRAS08}, which requires that 
\bea
\label{singvalcontrol}
(1-\delta) \|\beta_S\|_2^2 \le & \|X_S\beta_S\|_2^2 & \le (1+\delta) \|\beta_S\|_2^2, 
\eea
for $S\subset \{1,\ldots,n\}$ with $|S|=s^\prime$ 
and all $\beta \in \mathbb R^p$. This property is satisfied 
by high probability for most random matrices with i.i.d. entrees with variance $1/n$ \footnote{the $1/n$ assumption on the variance and standard concentration bounds 
	imply that the resulting random matrix has almost normalized columns} 
such as Gaussian or Rademacher variables and for $s^\prime \le C_{rip} \ n/\log(p)$, where 
the constant $C_{rip}$ depends on the distribution of the individual entrees.
RIP has been extensively used in signal processing after the emergence of 
the so-called Compressed Sensing paradigm \cite{Candes:ICM06}. 

A second assumption which is often considered is the Incoherence Condition, which 
requires that 
\bean 
\mu(X) & = & \max_{j\neq j^\prime=1\ldots,p} |\langle X_j,X_{j^\prime}\rangle|
\eean 
is small, e.g. $\mu(X)\le C_{\mu} /\log(p)$ as in \cite{CandesPlan:AnnStat09}, which 
is garanteed for random matrices with i.i.d. gaussian entrees with variance $1/n$ in 
the range $n \ge C_{ic} \log(p)^3$.

The main advantage of the Incoherence Condition over the Restricted Isometry Property 
is that it can be checked in $p(p-1)/2$ operations, whereas RIP is NP-hard to verify. The main relationship between the Incoherence Condition and RIP is that under the Incoherence Condition, (\ref{singvalcontrol}) holds, not for all, but for most supports 
$S\subset \{1,\ldots,n\}$ with cardinal $s^\prime$, where $s^\prime \le C_s \ p /(\|X\| \log(p))$, for 
some constant $C_s$ controlling the proportion of such supports \cite{CandesPlan:AnnStat09}. 
%RIP implique IC, mais IC implique RIP sous conditions

The objective of the present paper is to extend the analysis based on the Incoherence 
Condition to more general situations where $X$ may have a lot of very colinear columns. 
The main idea is to assume that the columns are drawn from a mixture model of $K$ 
clusters, and that the set of centers forms a (usually) much smaller matrix for which is it quite realistic to impose the Incoherence Condition. 
%mélange de gaussiennes : on a K gaussienne et on tire une colonne suivant une gaussienne avec proba pi_k

\section{Main results on  the LASSO}
In this section, we summarise the main results on the LASSO. 
\subsection{Background}

%Voir les notes de Rigollet 

%\href{http://www-math.mit.edu/~rigollet/PDFs/RigNotes17.pdf}{http://www-math.mit.edu/~rigollet/PDFs/RigNotes17.pdf}
We will study the linear regression model 
\begin{align}
\label{model}
    y = X\beta + z
\end{align}
where $y \in \mathbb{R}^{n}$ is the data, $X \in \mathbb{R}^{n\times p}$ is the matrix of explanatory variables, $\beta \in \mathbb{R}^{p}$ is the parameter of interest and $z$ is a centered gaussian i.i.d noise vector $\mathcal{N}\left(0,\sigma^{2}I\right)$.

The LASSO estimator of $\beta$ is defined by any $\hat{\beta}$ such that 
\begin{align}
\label{las}
\hat{\beta} = \displaystyle \argmin_{b \in \mathbb{R} ^{p}} \frac{1}{2}\Vert y - Xb\Vert_{2}^{2} + \lambda \Vert b\Vert_{1}.
\end{align}
%\begin{lem}
%If $X^{t}X$ is a positive definite matrix, (\ref{las}) admits a unique solution.
%\end{lem}
%\begin{proof}
%Let for any $\lambda > 0$, the function 
%\begin{align*}
 %   f_{\lambda}(b) = \frac{1}{2}\Vert y - Xb\Vert_{2}^{2} + \lambda \Vert b\Vert_{1}, \quad \quad \forall b \in \mathbb{R}^{p}
%\end{align*}
%By sum of two convex functions in $b$, the function $f_{\lambda}$ is convex. Further, the function $\frac{1}{2}\Vert y - Xb\Vert_{2}^{2}$ is strictly convex because its Hessian matrix is precisely $X^{t}X$ which is positive definite. Thus, $f_{\lambda}$ admits a unique solution to $\mathbb{R}^{p}$ for any $\lambda > 0$.
%\end{proof}
\begin{lem}
The LASSO estimator obeys
\begin{align}
\label{dua}
    \Vert X^{t}\left(y - X\hat{\beta}\right)\Vert_{\infty} \leq \lambda.
\end{align}
\end{lem}
%condition d’optimalité sur le lasso. Le gradient en norme infini est < lambda. si pas de pen, pour beta chapeau le gradient est nul

\subsection{Statistical viewpoint}
The LASSO estimator has been the subject of intense research in the recent years in the statistics community. Several results have been obtained about the mean squared error.  The first result below is about the case where no specific assumption is required about $X$.

\begin{theo}
Assume that the linear model (\ref{model}) holds where $z \sim \mathcal{N}(0,\sigma^{2})$. Moreover, assume that the columns of $X$ are normalized in such a way that $\max _{j}\left\Vert X_{j}\right\Vert_{2} \leq \sqrt{n}$. Then, the Lasso estimator $\hat{\theta}$ with regularization parameter
$$
\lambda=\sigma \sqrt{\frac{2 \log (2 p)}{n}}+\sigma \sqrt{\frac{2 \log (1 / \delta)}{n}}
$$
satisfies
$$
\frac{1}{n}\left\Vert X \hat{\beta}-X \beta^{*}\right\Vert_{2}^{2} \leq 4\lambda \left\Vert\beta^{*}\right\Vert_{1} \sigma
$$
with probability at least $1-\delta$.
\end{theo}

The next theorem states that when $X$ satisfies some incoherence-type assumption, more can be obtained for the LASSO estimator and the mean squared error decreases faster.

\begin{theo}
\label{fastrig}
Fix $n \geq 2$. Assume that the linear model holds where $z \sim$ $\mathcal{N}\left(0,\sigma^{2}\right)$. Moreover, assume that $\left\Vert\beta^{*}\right\Vert_{0} \leq s$ and that $X$ satisfies assumption $\operatorname{INC}(s)$. Then the Lasso estimator $\hat{\beta}$ with regularization parameter defined by
$$
\lambda=4 \sigma \sqrt{\frac{\log (2 p)}{n}}+4 \sigma \sqrt{\frac{\log (1 / \delta)}{n}}
$$
satisfies
$$
\left\Vert \hat{\beta}- \beta^{*}\right\Vert_{2}^{2} \lesssim s \sigma^{2} \log (2 p / \delta)
$$
with probability at least $1-\delta$.
\end{theo}
%INC = incoherence

In this paper, our goal is to extend this last result to the case where the design matrix has potentially many almost co-linear columns, using a mixture model as a generating model for the columns. 

\section{Our mixture model and a sketch of our main result}

\subsection{The mixture model}
\label{mix}

In order to relax the Incoherence Condition, one needs
a model for the design matrix $X$ allowing for a certain amount of almost parallel columns while keeping some of the algebraic structure in the same spirit 
as in (\ref{singvalcontrol}) for at least most supports indexing a subset of relevant covariates. In what follows, we study such a model, where the columns can be considered as drawn from a finite mixture of $K$ Gaussian distributions. 

An important parameter for the theoretical analysis is a separation index for the centers in the mixture model. This separation index we chose to study in this work is simply the coherence of the matrix of centers which is much smaller than the original design matrix $X$.

\subsubsection{Detailed presentation}

Let $K$ be the number of clusters of our model. Consider a 
matrix $\mathfrak{C}$ in $\R^{n\times K}$.
The columns $\mathfrak C_k$, $k=1,\ldots,K$ of the matrix $\mathfrak{C}$ are the "centers" of each cluster. 

In our model, 
\begin{itemize}
	\item the matrix $X$ is obtained as follows. 
\begin{itemize}
\item Choose $n_k$, $k=1,\ldots,K$.
\item Let $X_o\in \mathbb R^{n\times p}$ be a random matrix with independent random columns such that the $n_1$ columns follow the distribution $\phi_1$, the $n_2$ next columns follow the distribution $\phi_2$, etc, where 
\bean 
\phi_k(x) & = & \frac1{\left(2\pi \mathfrak{s}^2\right)^{\frac{n}2}} 
\exp\left(-\frac{\left\|x-\mathfrak{C}_k\right\|_2^2}{2\mathfrak{s}^2}\right).
\eean 
For each $j \in \{1,\ldots,p\}$, $k_j \in \{1,\ldots,K\}$ will denote the index of the Gaussian component 
from which columns $j$ was drawn, and $\mathcal J_k$ will denote the set of indices of the columns 
drawn from the $k^{th}$ Gaussian component. For any $S \subset \{1,\ldots,p\}$, $k_S$ will denote the subset of $\{1,\ldots,K\}$ indexing the centers of the distributions from which the columns of $X_S$ were drawn.
\item The matrix $X$ is obtained by a random permutation of the columns of $X_o$ and column-wise $\ell_2$-normalization.
\end{itemize} 

\item the support of $\beta$ will drawn at random as follows.
\begin{itemize} 
	\item We will assume that the support $T$ of the true regression vector $\beta$ is drawn in such a way that $T$ has the uniform distribution on the subsets of  $\{1,\ldots,K\}$ with cardinality equal to $s$.
\end{itemize}
\end{itemize}
\subsubsection{Best approximation of the class centers and projection of $\beta$}
For any index set $S\subset \{1,\ldots,p\}$, 
let ${\mathcal K}_S$ denote the list (with possible repetitions) 
\bean 
{\mathcal K}_S & = & \left\{ k_j \mid j \in S\right\}.
\eean 
For each $j\in \{1,\ldots,p\}$, the deviation of column $X_{o,j}$ from center 
$\mathfrak{C}_{k_{j}}$ will be denoted by $\epsilon_{j}$:
\bean 
\epsilon_j & = & X_{o,j}-\mathfrak{C}_{k_j}. 
\sim \mathcal N\left(0,\mathfrak{s}^2I\right). 
\eean
The matrix $E$ is defined as
\bean 
E & = & (\epsilon_{i,j})_{i\in \{1,\ldots,n\},\ j\in \{1,\ldots,p\}}.
\eean

For each $k \in \mathcal \{1, \ldots, K\}$, let $j_k^*$ be the best approximation of the 
center $\mathfrak{C}_k$ from the set of columns $X_j$, $j\in \mathcal J_k$, i.e. 
\bean 
j^*_k & = & \argmin_{j \in \mathcal J_k} \|X_j - \mathfrak{C}_{k}\|_2.
\eean 
Moreover, set 
\bean 
 T^* & = & \left\{j^*_k \mid k=1,\ldots, K \right\}.
\eean
Notice that in particular, $|T^*|=s^{*}$. 

Let $\beta^*$ be the vector defined as
\bea
\label{toto}
\mathfrak{C}_{\mathcal K_{T^*}}\beta^*_{T^*} & = & \mathfrak{C}_{\mathcal K_T}\beta_T.
\eea

A simple expression of $\beta^*$ can be obtained by taking 
\bea
\label{beta*}
\label{betastarbetaraw}
\beta^*_{j^*} & = & \sum_{j\in J_{k_{j^*}}\cap T} \beta_{j} 
\eea

for all $j^*\in T^*$. Moreover, this expression is unique whenever $X_{T^*}$ 
has rank equal to $s^*$\footnote{In Section \ref{prel}, we will show that $X_{T^*}$
is indeed non-singular with high probability under appropriate assumptions on $T$}.

%Notice that expression (\ref{betastarbetaraw}) can be written in matrix form as 
%\bea
%\label{betastarbeta}
%\beta^*_{T^*} & = & H \beta_T,
%\eea
%where $H$ is a matrix defined by 
%\bean 
%H_{j^*,j} & = & 
%\begin{cases}
%1 \textrm{ if } j \in J_{k_{j^*}} \\
%\\
%0 \textrm{ otherwise }
%\end{cases}
%\eean 

\subsection{Main result}

The following theorem shows a bound on the prediction error which is a function of the sparsity $s^{*}$, the number $n$ of observations, the number of columns $p$.
%and the intra-cluster variance $\mathfrak s$.

\begin{theo}\ \textit{(Sketch)}
	%\textcolor{red}{Trouver une forme simplifi\'ee ici !!!
	Let $\lb=2\sigma \sqrt{2\alpha \ \log(p)}$. 
	Assume that $X$ is drawn from the Gaussian mixture model of Section \ref{mix}. Then, for $p$ sufficiently large, with probability at least $1-C_{\alpha,n,\rho} (\rho^{\alpha}+p^{-\alpha})$, we have
	\begin{align*}
    \frac12 \|X\left(\hat{\beta} - \beta\right)\|_2^2  &\le  \frac{3}{2} \lambda s^{*} \frac{1}{1 - r^{*}_{\alpha,n,\rho}(r)}\left(\frac{3}{2}\ \lambda + \sqrt{3}\|X\left(\beta^{*} - \beta\right)\|_2\right)\\
    &\hspace{1cm}+ \frac12 \|X\left(\beta^{*} - \beta\right)\|_2^2.
    \end{align*}
    	with 
	\bea
\label{retoile}
r^*_{\alpha,n,\rho}(r) = r\left(\frac{1}{2} + 0.1\ \text{C}_{\alpha, n, \rho}\right)\left(2 + \frac{1}{2}r + 0.1\ r\ \text{C}_{\alpha, n, \rho}\right)
\eea
where $\text{C}_{\alpha,n,\rho} = \sqrt{\alpha} + \sqrt{\frac{\log(n)}{\log(\rho^{-1})}}$.
\end{theo}
%beta ici, c’est beta * ? beta c’est la vérité, et beta*
%c’est le meilleur beta qui utilise seulement les
%colonnes du mélange
%beta* c'est le modèle à partir des centres. beta* - beta c'est un biais.
%voir comment beta*-beta évolue

%\begin{theo}
%\label{mainrand}
%Assume that $X$ has the $(s,{s^*},\delta,\rho,\alpha)$ Random Sparse Approximation Property
%and that $X$ satisfies the Cand\`es-Plan Properties.
%Assume that the support $T$ of $beta$ is drawn uniformly at random over all index subsets of
%$\{1,\ldots,p\}$ with cardinal $s$. Then, we have
%\bean
%\frac12 \|X(\hat{\beta}-\beta)\|_2^2
%& \le  & {s^*}c(1+C)\lambda \left((1+C)\lambda + \sqrt{(1+\rho)\delta}\|X\beta\|_2\right) +\frac12 \delta \|X\beta\|_2^2
%\eean
%with probability at least $1-C\ p^{-\alpha}$.
%\end{theo}

\section{A general result and its proof}
Some parts of the proof closely follow the key arguments in the proof of   
\cite[Theorem 1.2]{CandesPlan:AnnStat09}, although many details of the needed adaptation are nontrivial. Our Theorem \ref{main} below contains the most general statement of our work. 

\subsection{A more general result}
We will require a set of assumptions that are described below. 
\subsubsection{Assumptions}
In the sequel $\alpha \ge 1$ and $r$ will denote a constant in $(0,1/2)$. 
The constants $\vartheta_*$ et $\nu$ will be specified in Assumptions \ref{varth} below. The 
constants $C_\mu$, $C_{\rm spar}$ et $C_{col}$ will be used in the Assumptions below:
\bean 
C_{\mu} & =  r/(1+\alpha),
\hspace{.1cm} C_{\rm spar} = r^2/((1+\alpha)e^2),
\hspace{.1cm} C_{col} = & \frac12 \left(\frac{\sqrt{2}}{\sqrt{(1-r)(1+\alpha)}}-(1+r)\right).
\eean 
Let $C_\chi$ denote a positive constant such that 
\bean 
\bP \left( \frac{\|G\|_2^2}{\mathfrak{s}^2} 
\le u^2 \right) & \le & C_\chi \left(\frac{u^2}{n}\right)^{n}
\eean 
where $G$ is a $n$-dimensional centered and unit-variance i.i.d. gaussian vector. 

We will make the following assumptions. 
\begin{ass}
	\label{ass1}
	The matrix $\mathfrak{C}$ has a small coherence, i.e. $\mu(\mathfrak{C})$ should satisfying
	\bea
	\mu(\mathfrak{C}) \le \frac{C_\mu}{\log(K)}
	\eea 
	for some positive constant $C_\mu$. 
\end{ass}
\begin{ass}
	\label{varth}
	The clusters must containt sufficiently many points, i.e. 
	there exists a positive real constant $\vartheta^*$ and a positive integer $\nu$ such that 
	\begin{align}
	\label{varthh}
	 \min_{j^*\in T^*}|\mathcal J_{k_{j^*}}|\ge \vartheta_* \log(p)^\nu.   
	\end{align} 
\end{ass}

\begin{ass}
	\label{sparss}
	The proxy $\beta^*$ must be sufficiently sparse, i.e.  
	\bean
	s^* & \le & K_o \ \frac{K}{\log K} \ \frac{C_{\rm spar}}{\|\mathfrak{C} \|^{2}} \label{ups}
	\eean 
	for some positive constant $C_{\rm spar}$ and $K_{0} \le \rho^{-1}$ for some $\rho \in (0,1)$. 
\end{ass}
\begin{ass}
\label{colC}
    The number of columns of $\mathfrak{C}$ satisfying
    \begin{align*}
        K \leq C_{\rm K} \log(p)
    \end{align*}
for some positive constant $C_{\rm K}$.
\end{ass}
\begin{ass}
	\label{n}
	One must have sufficiently many observations, i.e. 
	\bea
	n & \ge & \frac{\alpha+1}{c} \log(p)
	\eea
	for some positive constant $c$.
\end{ass}
\begin{comment}
\begin{ass}
	\label{ass0}
	The number of columns of $X$ is not too small: 
	\bean 
	p & \ge & \max \left\{ K_o, e^{\frac{e^{2-\log(n)}- \log(n)}{\alpha}}\right\}.
	\eean

\end{ass}
\end{comment}

\begin{rem}
	The number of observations is both controlled by Assumption \ref{n} and Assumption \ref{ass1} on the coherence of $\mathfrak{C}$. 
	For instance, if $\mathfrak{C}$ comes from a Gaussian i.i.d. random matrix, the 
	coherence will be of the order $\sqrt{\log(K)/n}$ as discussed in \cite[Section 1.1]{CandesPlan:AnnStat09} 
	and by Assumption \ref{ass1}, $n$ should be at least of the order $\log(K)^{3}/C_{\mu}^{2}$.
	Notice that this is still less than if $X$ itself had to satisfy the coherence 
	bound, which would imply that $n$ be of the order $\log(p)^3$. This demontrates the advantage of using our Gaussian Mixture framework over the standard framework based on incoherence on $X$. 
\end{rem}

\begin{ass}
	\label{ass3}
	The variance inside the clusters must be sufficiently small, so that the clusters are well
	separated. More precisely, we will require that  
	\bean 
	\mathfrak{s} & \le \min\left\{\cfrac{\alpha}{2\sqrt{n}};
	\frac{C_{\mathfrak{s},n,p}}{\sqrt{\log(\rho^{-1})} \ 
		\left(\sqrt{n}+\sqrt{\frac{\alpha+1}{c} \log(p)}\right)} \right\} 
	\eean 
	for any $C_{\mathfrak{s},n,p}$ such that 
	and 
	\bean
	C_{\mathfrak{s},n,p} 
	& \le & \min \Bigg\{
	0.1 \cdot \frac{r}{\sqrt{n\ 
			\left(\frac{\alpha\ (1-e^{-1})}{\vartheta_* \ C_\chi}\right)^{\frac1{n}}\left(\frac{1}{\log(p)^{\nu-1}}\right)^{\frac1{n}}}};
	\frac12\sqrt{\log(p)}
	\Bigg\}.
	\eean
\end{ass}

\begin{ass}
	The support of $\beta^*_{T^*}$ is suffiently generic. More precisely, 
	we will require that the support of $\beta*$ is random and uniformly distributed among 
	subsets of $\{1,\ldots,p\}$ with cardinal $s^*$. 
	The sign of $\beta^*_{T^*}$ is random with uniform distribution on $\{-1,1\}^{s^*}$. 
	\label{supp}
\end{ass}

\begin{rem}
	This last assumption is a transposition to the proxy $\beta^*$ 
	of the conditions on $\beta$ in \cite{CandesPlan:AnnStat09}. 
\end{rem}

\begin{ass}
	\label{asscol}
	Relationships between the constants. 
	\bean 
	C_{col} & \ge & e^2 (\alpha+1) \max \{\sqrt{C_{spar}},C_\mu \}.
	\eean 
	%and 
	%\textcolor{red}{\bean  
	%\left(C_{col}+(1+1.1\cdot r) \ C_{\mathfrak{s},n,p}
	%\right)  \le  \frac12 \sqrt{\frac{\log(p)\ (1-r^*)^2}{\left(\alpha \log(p)-\log(2)\right) 2}}.
	%\eean }
\end{ass}

\begin{ass}
\label{pbound}
    Assume that 
    \begin{align}
         p & \le  0.01 \cdot \rho^{1-\left(\alpha + 1\right) \log(K)^{2}}
    \end{align}
    for the same $\rho$ as in Assumption \ref{sparss}.
\end{ass}

\subsubsection{The general theorem}
The main result of this paper is the following theorem.
\begin{theo}
	\label{main}
	Let $\lb=2\sigma \sqrt{2\alpha \ \log(p)}$. 
	Assume that $X$ is drawn from the Gaussian mixture model of Section \ref{mix} with 
	$\mathcal K$ drawn uniformly at random among all possible index subsets 
	of $\{1,\ldots,K\}$ with cardinal $s^*$. Let 
	Assumptions \ref{ass1}, \ref{varth}, \ref{sparss}, \ref{colC}, \ref{n}, \ref{ass3}, 
	\ref{supp}, \ref{asscol} and  \ref{pbound} hold. 
	Then, for all $r\in (0,0.5)$, with probability at least $1-C_{\alpha,n,\rho} (\rho^{\alpha}+p^{-\alpha})$, we have

	\bean
     \|X\left(\hat{\beta} - \beta\right)\|_2^2  &\le  3 \lambda s^{*} \frac{1}{1 - r^{*}_{\alpha,n,\rho}(r)}\left(\frac{3}{2}\ \lambda + \sqrt{1 + r^{*}_{\alpha,n,\rho}(r)}\|X\left(\beta^{*} - \beta\right)\|_2\right)\\
    &\hspace{1cm}+ \|X\left(\beta^{*} - \beta\right)\|_2^2
    \eean
		%with $r_*=1.1\cdot r\ \left(1.1 + 0.11 \cdot r\right)$ and 
with 
\begin{align}
%\label{retoile}
r^*_{\alpha,n,\rho}(r) = r\left(\frac{1}{2} + 0.1\ \text{C}_{\alpha, n, \rho}\right)\left(2 + \frac{1}{2}r + 0.1\ r\ \text{C}_{\alpha, n, \rho}\right)
\label{errbnd}
\end{align}
where \footnote{$\rho\in (0,1)$ is introduced in Assumption \ref{pbound}} $\text{C}_{\alpha,n,\rho} = \sqrt{\alpha} + \sqrt{\frac{\log(n)}{\log(\rho^{-1})}}$ .
\end{theo}

\begin{rem}
Notice that our result is of fast rate type but includes new additional terms involving the approximation error $\beta^*-\beta$. More precisely, the right hand side in \eqref{errbnd} can be decomposed into two parts: 
\begin{itemize}
    \item the term 
    \begin{align*}
        \frac{9}{2n} \lambda^2 s^{*} \frac{1}{1 - r^{*}_{\alpha,n,\rho}(r)},  
    \end{align*} 
    which is similar to the "fast rate term" in the standard incoherent case of Theorem \ref{fastrig}.
    \item the term 
    \begin{align*}
    & \frac{3}{2}  \frac{\lambda s^{*}}{1 - r^{*}_{\alpha,n,\rho}(r)}\sqrt{1 + r^{*}_{\alpha,n,\rho}(r)}\|X\left(\beta^{*} - \beta\right)\|_2+ \|X\left(\beta^{*} - \beta\right)\|_2^2
    \end{align*}
    is not present in the standard analysis of the LASSO and depends on how well $\beta$ can be approximated by $\beta^*$, and depends on the model and more precisely $\mathfrak{C}$ and $\beta$. 
\end{itemize}
\end{rem}

\begin{rem}
Notice that the coefficient $\text{C}_{\alpha,n,\rho}$ can be made as small as necessary when $\rho$ is sufficiently larger than $n$. Thus, we can always pretend for the ease of the analysis, that $r^{*}_{\alpha,n,\rho}(r)$ is of the same order as $r$. 
\end{rem}

We now begin the proof of Theorem \ref{main}. 
\subsection{Preliminaries: Cand\`es and Plan's conditions}
\label{prel}
The following proposition will be much used in the arguments. 
\begin{prop}
\label{CPprop}
We have the following properties: 
\begin{enumerate}
\item 
\begin{align}
\label{ICformathfrakC}
\bP \left(\|\mathfrak{C}_{\mathcal K_T}^t\mathfrak{C}_{\mathcal K_T}-\Id_s \|\ge \rho_{\mathfrak{C}_{\mathcal K_T}} \right) 
 \le  \frac{216}{p^{\alpha}}.
\end{align}
\item 
\begin{align}
\label{invcond}
\bP \left( 
\left\|X_{T}^t X_{T}-I \right\| \ge r^{*}_{\alpha,n,p}(r) \right)  \le  
\frac{219}{p^\alpha}.
\end{align}
where $r^{*}_{\alpha,n,p}(r)$ is defined by (\ref{retoile})
\item 
\begin{align}
\label{orthcond}
\bP\left( \|X^t z\|_{\infty} \ge \sigma \sqrt{2\alpha \ \log(p)} \right)  \le  \frac1{p^\alpha}.
\end{align}
\item 
\begin{align}
\nonumber 
&  \left\|X_{{T}^{*^{c}}}^tX_{T^{*}}(X_{T^{*}}^t X_{T^{*}})^{-1}X^t_{T^{*}}z\right\|_\infty
+\lb \ \left\|X_{{T}^{*^{c}}}^tX_{T^{*}}(X_{T^{*}}^t X_{T^{*}})^{-1}\sgn\left(\beta^{*}_{T^{*}}\right)\right\|_\infty \\
&  \hspace{2cm} \le   %\textcolor{red}{\sigma \sqrt{1+1.1\cdot r(1.1+0.11\cdot r)}+\frac12 \ \lb}
\sigma \mathcal{C}_{1} + \lb \mathcal{C}_{2}
\label{compsize}
\end{align}
where $\mathcal{C}_{1}$ and $\mathcal{C}_{2}$ are defined by (\ref{c1}) and (\ref{c2}).
\end{enumerate}
\end{prop}

\begin{proof}
See Appendix \ref{CPconditions}.
\end{proof}

\subsection{The prediction bound}
By definition, the LASSO estimator satisfies
\bea
\frac12 \|y-X\hat{\beta}\|_2^2+\lambda \|\hat{\beta}\|_1 & \le &
\frac12 \|y-X{\beta^*}\|_2^2+\lambda \|{\beta^*}\|_1.
\eea
One may introduce $X\beta$ in this expression and obtain
\bean
\frac12 \|y-X\beta +X(\beta-\hat{\beta})\|_2^2+\lambda \|\hat{\beta}\|_1 & \le &
\frac12 \|y-X\beta+X(\beta-{\beta^*})\|_2^2+\lambda \|{\beta^*}\|_1,
\eean
from which we deduce
\bea
\label{boundtemp}
\frac12 \|X(\beta-\hat{\beta})\|_2^2 & \le  \langle y-X\beta,X(\hat{\beta}-{\beta^*})\rangle 
%\\
%\nonumber & & 
-\lambda \left(\|\hat{\beta}\|_1-\|{\beta^*}\|_1 \right)\\
& \quad +\frac12 \|X(\beta-{\beta^*})\|_2^2.\nonumber
\eea
Set ${h^*} := \hat{\beta}-{\beta^*}$. Using sparsity of ${\beta^*}$, we obtain that
${h^{*}_{{T^*}^c}}=\hat{\beta}_{{T^*}^c}-{\beta^{*}_{{T^*}^c}}=\hat{\beta}_{{T^*}^c}$.
Thus, we have
\bean
\|\hat{\beta}\|_1-\|{\beta^*}\|_1 & = & \|{\beta^*}+h^*\|_1-\|{\beta^*}\|_1 \\
& = & \|{\beta^*}_{{T^*}}+{h^*}_{{T^*}}\|_1+\|{\beta^*}_{{T^*}^c}
+{h^*}_{{T^*}^c}\|_1-\|{\beta^*}_{{T^*}}\|_1 \\
& = & \|{\beta^*}_{{T^*}}+{h^*}_{{T^*}}\|_1-\|{\beta^*}_{{T^*}}\|_1
+\|{h^*}_{{T^*}^c}\|_1.
\eean
Since, for any $b$ with no zero component, the gradient of $\|\cdot\|_1$ at $b$ is $\sgn(b)$, the
subgradient inequality gives
\bean
\|{\beta^*}_{{T^*}}+{h^*}_{{T^*}}\|_1  & \ge & \|{\beta^*}_{{T^*}}\|_1
+\left\langle \sgn\left({\beta^*}_{{T^*}}\right),{h^*}_{{T^*}} \right\rangle
\eean
and combining this latter inequality with (\ref{boundtemp}), we obtain
\begin{align}
\label{boundtemp2}
\frac12 \|X(\beta-\hat{\beta})\|_2^2  \le   \left\langle y-X\beta,X{h^*}\right\rangle
-\lambda \left\langle \sgn\left({\beta^*}_{{T^*}}\right),{h^*}_{{T^*}} \right\rangle
%\\\nonumber & & \hspace{.5cm}
-\lambda\|{h^*}_{{T^*}^c}\|_1+\frac12 \|X(\beta-{\beta^*})\|_2^2.
\end{align}
Set $\gamma:={\beta^*}-\beta$ and $h:=\hat{\beta}-\beta$. Using these notations,
equation (\ref{boundtemp2}) may be written
\bea
\label{boundtemp3}
\frac12 \|Xh\|_2^2 & \le &  \left\langle z,X{h^*}\right\rangle
-\lambda \left\langle \sgn\left({\beta^*}_{{T^*}}\right),{h^*}_{{T^*}} \right\rangle 
%\\\nonumber & & \hspace{.5cm} 
-\lambda\|{h^*}_{{T^*}^c}\|_1+\frac12 \|X \gamma\|_2^2.
\eea
Using the fact that
\bean
\langle X^tz,{h^*} \rangle  & = & \langle X_{T^*}^t z,{h^*}_{T^*}\rangle
+\langle X_{{T^*}^c}^t z,{h^*}_{{T^*}^c}\rangle
\eean
and the following majorization based on (\ref{orthcond})
\bean
\langle X_{{T^*}^c}^tz,{h^*}_{{T^*}^c}\rangle & \le & \|{h^*}_{{T^*}^c}\|_1 \|X^{t}_{{T^*}^c}z\|_\infty \\
& \le & \frac12 \ \lambda  \ \|{h^*}_{{T^*}^c}\|_1,
\eean
we obtain that
\bea
\frac12 \|Xh\|_2^2 & \le  & \langle v,{h^*}_{T^*} \rangle -(1-\frac12) \ \lambda \|{h^*}_{{T^*}^c}\|_1
+\frac12 \|X\gamma\|_2^2,
\eea
where $v := X_{T^*}^t z - \lambda \ \sgn\left({\beta^*}_{{T^*}}\right)$.

Now, observe that
\bean
\langle v,{h^*}_{T^*} \rangle & = &
\langle v,(X_{{T^*}}^tX_{{T^*}})^{-1}X_{{T^*}}^tX_{{T^*}} {h^*}_{T^*}\rangle \\
 & = & \langle (X_{{T^*}}^{t} X_{{T^*}})^{-1}v,X_{{T^*}}^tX_{{T^*}}{h^*}_{T^*}\rangle \\
 & = & \underbrace{ \langle (X_{{T^*}}^tX_{{T^*}})^{-1}v,X_{{T^*}}^tX{h^*}\rangle}_{A_1}
-\underbrace{\langle (X_{{T^*}}^tX_{{T^*}})^{-1}v,X_{{T^*}}^tX_{{T^*}^c}{h^*}_{{T^*}^c}\rangle}_{A_2}.
\eean
Let us begin by studying $A_2$. We have that
\bean
A_2 & \ge & -\|X_{{T^*}^c}^tX_{{T^*}}(X_{{T^*}}^tX_{{T^*}})^{-1}v\|_\infty \|h^{*}_{{T^*}^c}\|_1 \\
& \ge & -\|X_{{T^*}^c}^tX_{{T^*}}(X_{{T^*}}^tX_{{T^*}})^{-1}X_{T^*}^t z\|_\infty \|h^{*}_{{T^*}^c}\|_1 \\
& & - \lambda \ \|X_{{T^*}^c}^tX_{{T^*}}(X_{{T^*}}^tX_{{T^*}})^{-1}\sgn\left({\beta^*}_{{T^*}}\right)\|_\infty
\|{h^*}_{{T^*}^c}\|_1 \\
& \ge & - \left(\sigma \mathcal{C}_{1} + \lambda \mathcal{C}_{2}\right)\Vert h^{*}_{T^{*^{c}}}\Vert_{1}
\eean
%\textcolor{red}{-\left(\sigma \sqrt{1+1.1\cdot r\ \left(1.1 + 0.11 \cdot r\right)}+\frac12 \ \lb\right) 
%\ \|{h^*}_{{T^*}^c}\|_1}
by (\ref{compsize}). Thus
\bean
\langle v,{h^*}_{T^*} \rangle & \le & A_1 + \left(\sigma \mathcal{C}_{1} + \lambda \mathcal{C}_{2}\right)\Vert h^{*}_{T^{*^{c}}}\Vert_{1}
%\textcolor{red}{\left(\sigma \sqrt{1+1.1\cdot r\ \left(1.1 + 0.11 \cdot r\right)}+\frac12 \ \lb\right) 
%\ \|h_{{T^*}^c}\|_1}
\eean
%and since, by (\ref{logp}), 
%\bean 
%\textcolor{red}{\left(\sigma \sqrt{1+1.1\cdot r\ \left(1.1 + 0.11 \cdot r\right)}+\frac12 \ \lb\right)} & \le & \lb,\ \textcolor{red}{\text{Faudra revoir cela}}
%\eean 
and we deduce that
\bea
\frac12 \|Xh\|_2^2 & \le & A_1 + \left(\sigma \mathcal{C}_{1} + \lambda \mathcal{C}_{2} - \frac{1}{2}\lb\right)\Vert h^{*}_{T^{*^{c}}}\Vert_{1} +\frac12 \|X\gamma\|_2^2
\eea
Let us now bound $A_1$ from above. We have that
\bean
A_1 & \le &\underbrace{ \|X_{T^*}^t X{h^*}\|_\infty }_{B_1}\underbrace{\|(X_{T^*}^t X_{T^*})^{-1}v\|_1}_{B_2} \\
\eean
Firstly,
\bean
B_1 & \le & \|X_{T^*}^t (X{\beta^*}-y)\|_\infty +\|X_{T^*}^t(X\hat{\beta}-y)\|_\infty \\
& \le & \|X_{T^*}^t (X\gamma-z) \|_\infty +  \|X_{T^*}^t (y-X\hat{\beta})\|_\infty \\
& \le & \frac12 \lambda + \|X_{T^*}^t X\gamma \|_\infty + \lambda
\eean
where we used (\ref{orthcond}), and the optimality condition for the LASSO estimator ((\ref{dua})).
Secondly,
\bean
B_2 & \le & \sqrt{{s^*}}\|(X_{T^*}^t X_{T^*})^{-1}v\|_2 \\
& \le & \sqrt{{s^*}} \|(X_{T^*}^t X_{T^*})^{-1}\| \|v\|_2 \\
& \le & {s^*} \|(X_{T^*}^t X_{T^*})^{-1}\| \|v\|_\infty.
\eean
Moreover, (\ref{invcond}) gives $\Vert\left(X_{T^*}^t X_{T^*}\right)^{-1} \Vert \le \frac{1}{1 - r^{*}_{\alpha,n,p}(r)}$
%$\|(X_{T^*}^t X_{T^*})^{-1}\| \le 1.1\cdot r\ \left(1.1 + 0.11 \cdot r\right)$ and
\bean
\|v\|_\infty  \le  \|X_{T^*}^t z\|_\infty + \lambda  \le \frac32 \ \lambda
\eean
Thus, we obtain that
\bean
A_1 & \le & \frac{3}{2}\ \lambda\ s^{*}\ \frac{1}{1 - r^{*}_{\alpha,n,p}(r)}\left(\frac{3}{2}\ \lambda + \|X_{T^*}^t X\gamma \|_\infty\right)
%\textcolor{red}{{s^*}1.1\cdot r\ \left(1.1 + 0.11 \cdot r\right)
%\frac32\ \lambda\left(\frac32\ \lambda + \|X_{T^*}^t X\gamma \|_\infty\right)}
\eean
and thus,
\bean
\frac12 \|Xh\|_2^2 & \le   \frac{3}{2}\ \lambda\ s^{*}\ \frac{1}{1 - r^{*}_{\alpha,n,p}(r)}\left(\frac{3}{2}\ \lambda + \|X_{T^*}^t X\gamma \|_\infty\right) \\
& \hspace{1cm}+ \left(\sigma \mathcal{C}_{1} + \lambda \mathcal{C}_{2} - \frac{1}{2}\lb\right)\Vert h^{*}_{T^{*^{c}}}\Vert_{1} +\frac12 \|X\gamma\|_2^2.
%{s^*}1.1\cdot r\ \left(1.1 + 0.11 \cdot r\right)
%\frac32\ \lambda\left(\frac32\ \lambda + \|X_{T^*}^t X\gamma\|_\infty\right) +\frac12 \|X\gamma\|_2^2.
\eean
Since $\|X_{T^*}^t X\gamma\|_\infty\le \|X_{T^*}^t X\gamma\|_2$ and since 
%$\|X_{T^*}^t X\gamma\|_2\le \sqrt{1+1.1\cdot r\ \left(1.1 + 0.11 \cdot r\right)}\|X\gamma\|_2$\textcolor{red}{!!!}
$$\|X_{T^*}^t X\gamma\|_2\le \sqrt{1 + r^{*}_{\alpha,n,p}(r)}\|X\gamma\|_2,$$ we obtain
\bean
\frac12 \|Xh\|_2^2 & \le   \frac{3}{2} \lambda s^{*}\ \frac{1}{1 - r^{*}_{\alpha,n,p}(r)}\left(\frac{3}{2}\ \lambda + \sqrt{1 + r^{*}_{\alpha,n,p}(r)}\|X\gamma\|_2\right) \\
& + \left(\sigma \mathcal{C}_{1} + \lambda \mathcal{C}_{2} - \frac{1}{2}\lb\right)\Vert h^{*}_{T^{*^{c}}}\Vert_{1} +\frac12 \|X\gamma\|_2^2.
\eean
which completes the proof.

%\appendix

\section{Checking the Candes-Plan conditions}
\label{CPconditions}
The goal of this section is to Proposition \ref{CPprop} which gives a version of Cand\`es and 
Plan's conditions adapted to our Gaussian mixture model.
\subsection{Control of $\Vert E_{T^{*}}\Vert$}
Consider the matrix $E^{t}_{T^{*}}$, whose columns are independent. We would like to bound its operator norm.
\begin{lem}
\label{E*alpha}
Let the event 
\begin{align*}
    \mathcal{E}_{\alpha}^{*} = \bigcap_{j^*\in T^*}\left\{ \|E_{j^*}\|_2 \le 
\mathfrak{s} \ \sqrt{n\ \left(\frac{\alpha\ (1-e^{-1})}{\vartheta_* \ C_\chi}\right)^{\frac1{n}}
\left(\frac{1}{\log(p)^{\nu-1}}\right)^{\frac1{n}}} \right\}.
\end{align*}
Then, $\bP(\mathcal{E}^{*}_{\alpha}) \ge 1 - \rho^{\alpha}$.
\end{lem}
\begin{proof}
Using the independence of the $E_{j}$, $j\in \mathcal J_{k_{j^*}}$, we have 
\bean 
\bP \left( \|E_{j^*}\|_2 \ge u \right) & = & 
\bP \left( \min_{j\in \mathcal J_{k_{j^*}}} \|E_{j}\|_2 \ge u \right) \\
& = &  \prod_{j\in \mathcal J_{k_{j^*}}} \bP \left( \|E_{j}\|_2^2 \ge u^2 \right), \\
& \le &  \bP \left( \|E_{j}\|_2^2 \ge u^2 \right)^{\displaystyle \min_{j^*\in T^*} \ |\mathcal J_{k_{j^*}}|}.
\eean 
We also have 
\bean 
\bP \left( \|E_{j}\|_2^2 \ge u^2 \right) & = & 
1-\bP \left( \|E_{j}\|_2^2 \le u^2 \right).
\eean 
On the other hand, as is well known, we have 
\bean 
\bP \left( \frac{\|E_{j}\|_2^2}{\mathfrak{s}^2} 
\le u^2 \right) & \le & C_\chi \left(\frac{u^2}{n}\right)^{n}
\eean 
for some positive constant $C_\chi$. Thus, the union bound gives 
\bean 
\bP \left( \max_{j^*\in T^*} \ \|E_{j^*}\|_2 \ge u \right) 
& \le &  s^*\left(1-C_\chi \left(\frac{u^2}{n\ \mathfrak{s}^2}\right)^{n}\right)^{\displaystyle \min_{j^*\in T^*} \ |\mathcal J_{k_{j^*}}|}.
\eean 
Let us tune $u$ so that 
\bean 
s^* \left(1-C_\chi \left(\frac{u^2}{n\ \mathfrak{s}^2}\right)^{n}\right)^{\displaystyle \min_{j^*\in T^*} \ |\mathcal J_{k_{j^*}}|} & \le & \rho^{\alpha} 
\eean 
i.e.
\bean
u^2 & \ge & \frac{n\ \mathfrak{s}^2}{C_\chi^{\frac1{n}}}
\left(1-\left((s^*)^{-1}\rho^{\alpha}\right)^{\frac{1}{\displaystyle \min_{j^*\in T^*}|\mathcal J_{k_{j^*}}|}}\right)^{\frac1{n}} 
\eean 
and since $\min_{j^*\in T^*}|\mathcal J_{k_{j^*}}|\ge \vartheta_* \log(p)^\nu$ by (\ref{varthh}),
\bea 
u^2 & \ge & \frac{n\ \mathfrak{s}^2}{C_\chi^{\frac1{n}}}
\left( 1-\exp\left(-\frac{\alpha}{\vartheta_* \log(p)^{\nu-1}}-\frac{\log(s^*)}{\vartheta_* \log(p)^\nu}\right) \right)^{\frac1{n}}.
\eea 
On $(0,1)$, we have 
\bean 
\exp(-z) & \le & 1-(1-e^{-1})z
\eean 
and thus, 
\bean
u^2 & \ge & n\ \mathfrak{s}^2\ \left(\frac{\alpha\ (1-e^{-1})}{\vartheta_* \ C_\chi}\right)^{\frac1{n}}
\left(\frac{1}{\log(p)^{\nu-1}}\right)^{\frac1{n}},
\eean 
from which the desired estimate follows. 
\end{proof}
\begin{lem}
\label{normET*}
We have
\bean 
\bP\left(\left\|E_{T^*}^t\right\| 
\ge \mathfrak{s} K_{n,s^*}  \mid \mathcal E^*_{\alpha}\right) & \le & \frac{2}{p^{\alpha}}
\eean  
where
\begin{align}
\label{kval}
%K_{n,s^{*}} = \sqrt{n\left(\alpha\log(p) + \log(n)\right)  
%\left(\frac{\alpha\ (1-e^{-1})}{\vartheta_* \ C_\chi}\right)^{\frac1{n}}
%\left(\frac{1}{\log(p)^{\nu-1}}\right)^{\frac1{n}}}. 
K_{n,s^*}  = \sqrt{n \left(\alpha\log(\rho^{-1}) + \log(n)\right)
\left(\frac{\alpha\ (1-e^{-1})}{\vartheta_* \ C_\chi}\right)^{\frac1{n}}
\left(\frac{1}{\log(p)^{\nu-1}}\right)^{\frac1{n}}}.
\end{align}
\end{lem}

\begin{proof}
Let us first notice that since $\|E_{T^*}\|=\|E_{T^*}^t\|$, we can write 
\bean 
\|E_{T^*}^t\| & = & \sqrt{\|E_{T^*} E_{T^*}^t\|} \\
& = & \sqrt{\left\|\sum_{j^*\in T^*} E_{j^*} E_{j^*}^t\right\|}
\eean 
This latter expression is well suited for our problem, since it is the norm  
of the sum of independent positive semi-definite random matrices.
%, for which the Matrix Chernov inequality of Section \ref{troppchern} applies. 
In order to apply this inequality, we need a bound on the norm of each summand. 
By Lemma \ref{E*alpha}, on $\mathcal E^*$, we have 
\bean 
\left\|E_{j^*} E_{j^*}^t \right\|_{2} & = & \left\|E_{j^*} \right\|_2^2 \\
& \le & \mathfrak{s}^2 \ n\ \left(\frac{\alpha\ (1-e^{-1})}{\vartheta_* \ C_\chi}\right)^{\frac1{n}}
\left(\frac{1}{\log(p)^{\nu-1}}\right)^{\frac1{n}}. 
\eean 
We also need a bound on the norm of the expectation. We have 
\bean 
\left\|\mathbb E \left[ \sum_{j^*\in T^*} E_{j^*} E_{j^*}^t 
\mid \mathcal E_\alpha^*\right]\right\| & = & 
\left\|\sum_{j^*\in T^*} \mathbb E \left[ E_{j^*} E_{j^*}^t 
\mid \mathcal E_\alpha^*
\right]\right\|. 
\eean  
Due to rotational invariance, we have that the law of $E_{j^*}$ is the 
same as the law of $D(\zeta) E_{j^*}$, where $\zeta_1,\ldots,\zeta_n$ are 
i.i.d. Rademacher $\pm 1$ random variables independent from $E_{j^*}$. 
Thus, for $i\neq i'$, 
\bea
\mathbb E \left[\zeta_i E_{i,j^*}\zeta_{i^\prime} E_{i^\prime,j^*} 
\mid \mathcal E_\alpha^* \right]
& = & 
\mathbb E \left[ \mathbb E\left[ \zeta_i E_{i,j^*}\zeta_{i^\prime} E_{i^\prime,j^*} \mid E_{i,j^*}, 
\ E_{i^\prime,j^*} \right]| \mathcal E_\alpha^*\right]  \nonumber \\
\nonumber \\
& = & 0.
\label{diagop}
\eea 
On the other hand, we have the following result.  
\begin{lem}
\label{Eij*2}
We have 
\bean 
\mathbb E \left[E_{i,j^*}^2\mid \mathcal E_\alpha^*\right]
& \le & \mathfrak{s}^2 \  \left(\frac{\alpha\ (1-e^{-1})}{\vartheta_* \ C_\chi}\right)^{\frac1{n}}
\left(\frac{1}{\log(p)^{\nu-1}}\right)^{\frac1{n}}.
\eean 
\end{lem}
\begin{proof}
Due to rotational invariance of the law of $E_{j^*}$ and the event $\mathcal E_\alpha^*$, we have 
\bean 
\mathbb E \left[E_{1,j^*}^2\mid \mathcal E_\alpha^* \right] 
& = \cdots = & \mathbb E \left[E_{n,j^*}^2
\mid \mathcal E_\alpha^*\right].
\eean  
               Therefore, 
\bean 
\mathbb E \left[E_{i,j^*}^2\mid \mathcal E_\alpha^* \right] & \le & 
\frac1{n} \ \mathbb E \left[\sum_{i=1}^n E_{i,j^*}^2
\mid \mathcal E_\alpha^*  \right]
\eean 
and by the definition of $\mathcal E_\alpha^*$, 
\bean 
\mathbb E \left[E_{i,j^*}^2\mid \mathcal E_\alpha^* \right] & \leq &  
\mathfrak{s}^2 \  \left(\frac{\alpha\ (1-e^{-1})}{\vartheta_* \ C_\chi}\right)^{\frac1{n}}
\left(\frac{1}{\log(p)^{\nu-1}}\right)^{\frac1{n}}.
\eean 
\end{proof}

Based on this lemma, and the fact that the matrix 
\bean 
\mathbb E \left[ \sum_{j^*\in T^*} E_{j^*} E_{j^*}^t 
\mid \mathcal E_\alpha^*\right],
\eean 
is diagonal\,  by (\ref{diagop}), we obviously obtain that 
\bean 
\left\|\mathbb E \left[ \sum_{j^*\in T^*} E_{j^*} E_{j^*}^t 
\mid \mathcal E_\alpha^*\right]\right\| & \leq & 
\mathfrak{s}^2 \  \left(\frac{\alpha\ (1-e^{-1})}{\vartheta_* \ C_\chi}\right)^{\frac1{n}}
\left(\frac{1}{\log(p)^{\nu-1}}\right)^{\frac1{n}}.
\eean 
With the bound on the norm of the expectation and on the variance in hand
%,we are now ready to apply the Matrix Chernov inequality 
and obtain 
\bea 
& & \nonumber \bP \left(\left\| \sum_{j^*\in T^*} E_{j^*} E_{j^*}^t \right\| \ge u
\mid \mathcal E_\alpha^* \right) \\
& & \hspace{2cm} \le n \left( \frac{e \ \mathfrak{s}^2 \  \left(\frac{\alpha\ (1-e^{-1})}{\vartheta_* \ C_\chi}\right)^{\frac1{n}}
\left(\frac{1}{\log(p)^{\nu-1}}\right)^{\frac1{n}}}{u}\right)^{\frac{u}{\mathfrak{s}^2 \ n \ \left(\frac{\alpha\ (1-e^{-1})}{\vartheta_* \ C_\chi}\right)^{\frac1{n}}
\left(\frac{1}{\log(p)^{\nu-1}}\right)^{\frac1{n}}}}. 
\eea 
Let us finally tune $u$ so that the right hand side term is less than $\rho^{\alpha}$, i.e. 
\bean 
& &  \log\left( \frac{e \ \mathfrak{s}^2 \  \left(\frac{\alpha\ (1-e^{-1})}{\vartheta_* \ C_\chi}\right)^{\frac1{n}}
\left(\frac{1}{\log(p)^{\nu-1}}\right)^{\frac1{n}}}{u}\right) \\
& & \hspace{2cm} \le - \ \frac{\mathfrak{s}^2 \ n \ \left(\frac{\alpha\ (1-e^{-1})}{\vartheta_* \ C_\chi}\right)^{\frac1{n}}
\left(\frac{1}{\log(p)^{\nu-1}}\right)^{\frac1{n}}}{u}\ \left(\alpha\log(\rho^{-1}) + \log(n)\right).\, 
\eean 
Take 
\bea
\label{uvalue}
u & = &  \mathfrak{s}^2 \ n \left(\alpha\log(\rho^{-1}) + \log(n)\right) 
\left(\frac{\alpha\ (1-e^{-1})}{\vartheta_* \ C_\chi}\right)^{\frac1{n}}
\left(\frac{1}{\log(p)^{\nu-1}}\right)^{\frac1{n}}.\, \,
\eea 
Moreover, 
the value of $u$ given by (\ref{uvalue}) is less than or equal to 
$\mathfrak{s}^2 \ K_{n,s^*}^2$ with
\begin{align}
    K_{n,s^*} & = \sqrt{n \left(\alpha\log(\rho^{-1}) + \log(n)\right)
\left(\frac{\alpha\ (1-e^{-1})}{\vartheta_* \ C_\chi}\right)^{\frac1{n}}
\left(\frac{1}{\log(p)^{\nu-1}}\right)^{\frac1{n}}}.
\end{align}
%given by (\ref{K})
This completes the proof. 
\end{proof}

\subsection{Important properties of $\mathfrak{C}$}
The invertibily condition for (\ref{ICformathfrakC}) is a direct consequence of 
\cite{Tropp:CRAS08}. An alternative approach, based on the Matrix 
Chernov inequality is proposed in \cite{ChretienDarses:SPL12}, with improved constants. 
We have in particular 
\begin{theo}\cite[Theorem 1]{ChretienDarses:SPL12}
\label{invertC}
Let $r\in(0,1/2)$, $\alpha \ge 1$. Let Assumptions \ref{ass1} and \ref{sparss} hold with
\begin{align}
{\rm C}_{spar}  =  \frac{r^2}{4(1+\alpha) e^2}\ .
\end{align}
With $\mathcal K\subset \left\{1,\ldots,K\right\}$ 
%$T:\Omega \to \{I\subset \left\{1,\ldots,p\right\}, |I|=s\}$
chosen randomly from the uniform distribution among subsets with cardinality $s^*$, 
the following bound holds:
\bea \label{sing}
\bP \left(\|\mathfrak{C}_{\mathcal K}^t\mathfrak{C}_{\mathcal K}-\Id_s \|\ge r \right) 
& \le & \frac{216}{p^{\alpha}}.
\eea
\end{theo}
Moreover, the following property will also be very useful. 
\begin{lem} (Adapted from \cite[Lemma 5.3]{ChretienDarses:SPL12})
\label{cd12}
If $v^2 \ge e\ s^*\ \|\mathfrak{C}\|/ K_o $, we have 
\bean 
\bP \left( \max_{k\in \mathcal K^c} 
\left\|\mathfrak{C}_{\mathcal K}^t \mathfrak{C}_{k}\right\|_2\ge \frac{v}{1-r} \right) & \le & 
K_o \left(e\frac{s^*\ \|\mathfrak{C}\|^2 }{K_o\ v^2} \right)^{\frac{v^2}{\mu(\mathfrak{C})^2}}.
\eean 
\end{lem}
Based on this lemma, we easily get the following bound. 
\begin{lem}
\label{Ccol}
Take $C_{col} \ge\sqrt{ e^2 (\alpha+1)} \max \{\sqrt{C_{spar}},C_\mu \}$. Then, we have 
\bea 
\bP \left( \max_{k\in \mathcal K^c} 
\left\|\mathfrak{C}_{\mathcal K}^t \mathfrak{C}_{k}\right\|_2\ge \frac{C_{col}\cdot \sqrt{\log(\rho^{-1})}}{1-r} \right) 
& \le & 
 \frac1{\rho^{1  - \left(\alpha + 1\right) \log(K)^{2}}}.
\eea  
\end{lem}
\begin{proof}
Taking $v=C_{col}\cdot \sqrt{\log(\rho^{-1})}$, we obtain from Lemma \ref{cd12}
\bean 
\bP \left( \max_{k\in \mathcal K^c} 
\left\|\mathfrak{C}_{\mathcal K}^t \mathfrak{C}_{k}\right\|_2\ge \frac{C_{col}\cdot \sqrt{\log(\rho^{-1})}}{1-r} \right) 
& \le & 
K_o \left(e\frac{s^*\ \|\mathfrak{C}\|^2 }{K_o\ C_{col}^2\cdot \log(\rho^{-1})} 
\right)^{\frac{C_{col}^2}{C_\mu^2}\log(\rho^{-1})\cdot\log(K)^{2}}.
\eean 
Using Assumption \ref{sparss}, this gives
\begin{align*}
& \bP \left( \max_{k\in \mathcal K^c} 
\left\|\mathfrak{C}_{\mathcal K}^t \mathfrak{C}_{k}\right\|_2\ge \frac{C_{col}\cdot \sqrt{\log(\rho^{-1})}}{1-r} \right) \le \\
& \hspace{2cm} K_o \left(e\ K\ \frac{C_{spar}}{K_{o}\ \log(K) \ C_{col}^2\cdot \log(\rho^{-1}) }
\right)^{\frac{C_{col}^2}{C_\mu^2}\log(\rho^{-1})\cdot \log(K)^{2}}
\end{align*}
and using Assumption \ref{colC}, we have
\bean
\bP \left( \max_{k\in \mathcal K^c} 
\left\|\mathfrak{C}_{\mathcal K}^t \mathfrak{C}_{k}\right\|_2\ge \frac{C_{col}\cdot \sqrt{\log(\rho^{-1})}}{1-r} \right) 
&\le & 
K_o \left(e \  \frac{C_{spar}\ C_{K}}{K_{o}\ \log(K) \ C_{col}^2} \right)^{\frac{ C_{col}^2}{C_\mu^2}\log(\rho^{-1})\cdot \log(K)^{2}}.
\eean
Since $C_{col} \ge \sqrt{ e^2 (\alpha+1)} \max \{\sqrt{C_{spar}},C_\mu \}$, we get
\begin{align*}
& K_o \left(e \  \frac{C_{spar}\ C_{K}}{K_{o}\ \log(K) \ C_{col}^2} \right)^{\frac{ C_{col}^2}{C_\mu^2}\log(\rho^{-1})\cdot \log(K)^{2}} \\
& \hspace{.3cm}\leq K_{o} \left(e\   \frac{ C_{K}}{K_{o} \ \log(K) \ e^{2}\left(\alpha + 1\right)}\right)^{e^{2}\left(\alpha + 1\right)\log(\rho^{-1})\cdot \log(K)^{2}}.
\end{align*}
%\bean 
%K_o \left(e\ p \ \frac{C_{spar}}{C_{col}^2} \right)^{\frac{C_{col}^2}{C_\mu^2}\log(p)} 
%& \le &  K_o \left(\frac{e}{\alpha+1} \right)^{(\alpha+1)\ \log(p)}\ {\color{red}{\text{???}}}
%\eean 
%\begin{align*}
   % K_o \left(e \ \frac{C_{spar}}{C_{col}^2} \right)^{\frac{C_{col}^2}{C_\mu^2}\log(p)} \leq K_{o}\left(\frac{e}{\alpha + 1}\right)^{(\alpha + 1)\log(p)}
%\end{align*}
Thus, we have 
\bean 
\bP \left( \max_{k\in \mathcal K^c} 
\left\|\mathfrak{C}_{\mathcal K}^t \mathfrak{C}_{k}\right\|_2\ge \frac{C_{col}\cdot \sqrt{\log(\rho^{-1})}}{1-r} \right) 
& \le & 
K_{o}\ \frac1{(\rho^{-1})^{\left(\alpha + 1\right) \log(K)^{2}}}\cdot \frac1{K_{o}^{(\alpha + 1)\log(\rho^{-1})\log(K)^{2}}}.
\eean 
%and since, by Assumption \ref{ass0}, \textcolor{red}{$K_o\le p$}, we obtain that
and since $K_{0} \le \rho^{-1}$ by Assumption \ref{sparss}, we obtain
\bean 
\bP \left( \max_{k\in \mathcal K^c} 
\left\|\mathfrak{C}_{\mathcal K}^t \mathfrak{C}_{k}\right\|_2\ge \frac{C_{col}\cdot \sqrt{\log(\rho^{-1})}}{1-r} \right) 
& \le & 
 \frac1{\rho^{1  - \left(\alpha + 1\right) \log(K)^{2}}}.
\eean 
\end{proof}

\subsection{Similar properties for $X_{T^*}$}
\subsubsection{Control of $\left\|X_{T^*}^t X_{T^*}-I \right\|$}
We have 
\bean 
\sigma_{\min}\left(X_{T^*}^tX_{T^*} \right) & = & 
\sigma_{\min}\left(\left(\mathfrak{C}_{\mathcal K_{T^*}}+E_{T^*} \right)^t D_{*}^2
\left(\mathfrak{C}_{\mathcal K_{T^*}}+E_{T^*}\right) \right)
\eean 
%where (see Step 1 in the proof of Proposition \ref{betastar}) $D_*$ is a diagonal matrix whose 
where $D_*$ is a diagonal matrix whose diagonal elements are indexed by $T^*$ and are defined by 
\bean
D_{*,j^*,j^*} & = & \frac{1}{\left\|\mathfrak{C}_{k_{j^*}}+E_{j^*}\right\|_2}, 
\eean 
for $j^*\in T^*$. By the definition of $\mathcal E_\alpha^*$, we have 
\bean 
\sigma_{\min}(D_{*}) & \ge & \frac{1}{1+\mathfrak{s} \sqrt{n \left(\frac{\alpha(1-e^{-1})}
{\vartheta_* \ C_\chi} \right)^{\frac1{n}} \left(\frac1{\log(p)^{\nu-1}}\right)^{\frac1{n}}} }.
\eean 
and 
\bean 
\sigma_{\max}(D_{*}) & \le & \frac{1}{1-\mathfrak{s} \sqrt{n \left(\frac{\alpha(1-e^{-1})}
{\vartheta_* \ C_\chi} \right)^{\frac1{n}} \left(\frac1{\log(p)^{\nu-1}}\right)^{\frac1{n}}} }.
\eean 
By the triangular inequality, 
\bean 
\sigma_{\min}\left(X_{T^*}^tX_{T^*} \right) & \ge & 
\sigma_{\min} \left(\mathfrak{C}_{\mathcal K_{T^*}}^tD_*^2\mathfrak{C}_{\mathcal K_{T^*}}\right)
-2 \left\|\mathfrak{C}_{\mathcal K_{T^*}}^tD_*^2E_{T^*}\right\|-\left\|E_{T^*}^tD_*^2 E_{T^*} \right\|\ \\
& \ge &  \frac{1-r}{\left(1+\mathfrak{s} \sqrt{n \left(\frac{\alpha(1-e^{-1})}
{\vartheta_* \ C_\chi} \right)^{\frac1{n}} \left(\frac1{\log(p)^{\nu-1}}\right)^{\frac1{n}}} \right)^2} \\
& & \hspace{0cm} -
\frac{2\sqrt{1+r} \ \left\|E_{T^*}\right\|+\left\|E_{T^*}\right\|^2}
{\left(1-\mathfrak{s} \sqrt{n \left(\frac{\alpha(1-e^{-1})}
{\vartheta_* \ C_\chi} \right)^{\frac1{n}} \left(\frac1{\log(p)^{\nu-1}}\right)^{\frac1{n}}} 
\right)^2}.
\eean 
and 
\bean 
\sigma_{\max}\left(X_{T^*}^tX_{T^*} \right) & \le & 
\left\|\mathfrak{C}_{\mathcal K_{T^{*}}}^tD_*^2\mathfrak{C}_{\mathcal K_{T^{*}}}\right\|
+2  \left\|\mathfrak{C}_{\mathcal K_{T^{*}}}^tD_*^2 E_{T^*}\right\|+\left\|E_{T^*}^tD_*^2 E_{T^*} \right\| \\
& \le & \frac{(1+r) +2\sqrt{1+r} \ \left\|E_{T^*}\right\|+\left\|E_{T^*}\right\|^2}
{\left( 1-\mathfrak{s} \sqrt{n \left(\frac{\alpha(1-e^{-1})}
{\vartheta_* \ C_\chi} \right)^{\frac1{n}} \left(\frac1{\log(p)^{\nu-1}}\right)^{\frac1{n}}} \right)^2}.
\eean 
Moreover, using Theorem \ref{invertC} and Lemma \ref{normET*}, we obtain 
\bean 
\bP \left( 
\left\|X_{T^*}^t X_{T^*}-I \right\| \ge r^* \mid \mathcal E^*_{\alpha}\right) & \le & 
\frac{218}{p^{\alpha}}
\eean 
with $r^*$ given by 
\bea
\nonumber r^* & = & \max \Bigg\{ 
\frac{(1+r) +2\sqrt{1+r} \ \mathfrak{s} K_{n,s^*} +\mathfrak{s}^2 K_{n,s^*}^2}
{\left( 1-\mathfrak{s} \sqrt{n \left(\frac{\alpha(1-e^{-1})}
{\vartheta_* \ C_\chi} \right)^{\frac1{n}} \left(\frac1{\log(p)^{\nu-1}}\right)^{\frac1{n}}} \right)^2}-1; \\
\label{r*} & & 1- 
\Bigg(\frac{1-r}{\left(1+\mathfrak{s} \sqrt{n \left(\frac{\alpha(1-e^{-1})}
{\vartheta_* \ C_\chi} \right)^{\frac1{n}} \left(\frac1{\log(p)^{\nu-1}}\right)^{\frac1{n}}} \right)^2} \\
\nonumber & & \hspace{0cm} -
\frac{2\sqrt{1+r} \ \mathfrak{s} K_{n,s^*}+\mathfrak{s}^2 K_{n,s^*}^2}
{\left(1-\mathfrak{s} \sqrt{n \left(\frac{\alpha(1-e^{-1})}
{\vartheta_* \ C_\chi} \right)^{\frac1{n}} \left(\frac1{\log(p)^{\nu-1}}\right)^{\frac1{n}}} 
\right)^2}
\Bigg)
\Bigg\}.
\eea
Using (\ref{kval}) and Assumption (\ref{ass3}), we have 
\bean 
\mathfrak{s} \ K_{n,s^*} & \le 
C_{\mathfrak{s},n,p} \frac{\sqrt{\alpha\ \left(\frac{\alpha\ (1-e^{-1})}{\vartheta_* \ C_\chi}\right)^{\frac1{n}}\left(\frac{1}{\log(p)^{\nu-1}}\right)^{\frac1{n}}}}
{\left(1+\sqrt{\frac{\frac{\alpha+1}{c} \log(p)}{n}}\right)} \\
& +\frac{C_{\mathfrak{s},n,p}}{\sqrt{\log(p)}}  \frac{\sqrt{\log(n)\ \left(\frac{\alpha\ (1-e^{-1})}{\vartheta_* \ C_\chi}\right)^{\frac1{n}}\left(\frac{1}{\log(\rho^{-1})^{\nu-1}}\right)^{\frac1{n}}}}
{\left(1+\sqrt{\frac{\frac{\alpha+1}{c} \log(p)}{n}}\right)}, 
\eean 
and thus, 
\bean
\mathfrak{s} \ K_{n,s^*} & \le & 
C_{\mathfrak{s},n,p} \sqrt{\alpha\ \left(\frac{\alpha\ (1-e^{-1})}{\vartheta_* \ C_\chi}\right)^{\frac1{n}}\left(\frac{1}{\log(p)^{\nu-1}}\right)^{\frac1{n}}} \\
& & \hspace{1cm}+ C_{\mathfrak{s},n,p}\sqrt{\frac{\log(n)}{\log(\rho^{-1})}} \sqrt{\left(\frac{\alpha\ (1-e^{-1})}{\vartheta_* \ C_\chi}\right)^{\frac1{n}}\left(\frac{1}{\log(p)^{\nu-1}}\right)^{\frac1{n}}} , \nonumber \\
& = & C_{\mathfrak{s},n,p}\sqrt{\left(\frac{\alpha\ (1-e^{-1})}{\vartheta_* \ C_\chi}\right)^{\frac1{n}}\left(\frac{1}{\log(p)^{\nu-1}}\right)^{\frac1{n}}}\left(\sqrt{\alpha} + \sqrt{\frac{\log(n)}{\log(\rho^{-1})}} \right)\\
& \le & 0.1\ r\left(\sqrt{\alpha} + \sqrt{\frac{\log(n)}{\log(\rho^{-1})}}\right)
%& \le & \frac{0.1 \cdot r}{\sqrt{n}}\left(\sqrt{\alpha} + \sqrt{\frac{\log(n)}{\log(\rho^{-1})}}\right)  %\\
%& \le & 0.1\cdot r\ \left(\sqrt{\alpha} + \sqrt{\frac{\log(n)}{\log(p)}}\right)
\eean 
On the other hand, 
\bean 
\mathfrak{s} \sqrt{n \left(\frac{\alpha(1-e^{-1})}
{\vartheta_* \ C_\chi} \right)^{\frac1{n}} \left(\frac1{\log(p)^{\nu-1}}\right)^{\frac1{n}}}
& \le & 
\frac{C_{\mathfrak{s},n,p}}{\sqrt{\log(\rho^{-1})}} \ 
\frac{\sqrt{ \left(\frac{\alpha(1-e^{-1})}
{\vartheta_* \ C_\chi} \right)^{\frac1{n}} \left(\frac1{\log(p)^{\nu-1}}\right)^{\frac1{n}}}}
{\left(1+\sqrt{\frac{\frac{\alpha+1}{c} \log(p)}{n}}\right)} \\
& \le & \frac{C_{\mathfrak{s},n,p}}{\sqrt{\log(\rho^{-1})}} \
\sqrt{ \left(\frac{\alpha(1-e^{-1})}
{\vartheta_* \ C_\chi} \right)^{\frac1{n}} \left(\frac1{\log(p)^{\nu-1}}\right)^{\frac1{n}}} 
\eean 
which, by Assumption \ref{ass3}, gives 
\bea
\mathfrak{s} \sqrt{n \left(\frac{\alpha(1-e^{-1})}
{\vartheta_* \ C_\chi} \right)^{\frac1{n}} \left(\frac1{\log(p)^{\nu-1}}\right)^{\frac1{n}}}
%& \le & \frac{0.1\cdot r}{\log(p)}.
& \le & 0.1\ r
\eea 
Summing up, we get 
\bea
%\nonumber r^* & \le & 
%\frac{(1+r) +2\sqrt{1+r} \cdot \frac{0.1 \cdot r}{\sqrt{n}} \left(\sqrt{\alpha} + \sqrt{\frac{\log(n)}{\log(p)}}\ \right)
%+\frac{0.01 \cdot r^2}{n}\left(\sqrt{\alpha} + \sqrt{\frac{\log(n)}{\log(p)}}\ \right)^{2}}
%{\left( 1-\frac{0.1 \cdot r}{\log(p)} \right)^2}-1 \\
%\nonumber &=& \left(\frac{\sqrt{1+r} +\frac{0.1 \cdot r}{\sqrt{n}}\left(\sqrt{\alpha} + \sqrt{\frac{\log(n)}{\log(p)}}\ \right)}{1-\frac{0.1 \cdot r}{\log(p)}} \right)^{2} - 1\\
%& \le & r^*_{\alpha,n,p}(r) 
\nonumber r^{*} & \le &  (1+r) +2\sqrt{1+r} \cdot 0.1 \cdot r \left(\sqrt{\alpha} + \sqrt{\frac{\log(n)}{\log(\rho^{-1})}}\ \right)
+0.01 \cdot r^{2}\left(\sqrt{\alpha} + \sqrt{\frac{\log(n)}{\log(\rho^{-1})}}\ \right)^{2} - 1\\
& = & \left(\sqrt{1 + r} + 0.1\ r\ \text{C}_{\alpha, n, \rho}\right)^{2} -1\\
& \le & r^*_{\alpha,n,\rho}(r) \nonumber
%{\left( 1-\frac{0.1 \cdot r}{\log(p)} \right)^2
\eea
with 
%\textcolor{red}{
%\bea 
%\label{retoile}
%r^*_{\alpha,n,p}(r) &=& 
%r\left[\left(\frac{1}{2} + 0.1\ \text{C}_{\alpha, n, p}\right)\left(1 +  %\frac{0.1}{\log(p)}\right)+ \frac{0.1}{\log(p)}\right]\\
%&& \hspace{.5cm} \times  \left[2 + r\left(\left(\frac{1}{2} + 0.1\ \text{C}_{\alpha, n, p}\right)\left(1 +  \frac{0.1}{\log(p)}\right)+ %\frac{0.1}{\log(p)}\right)\right]\nonumber
%\eea}
\bea
%\label{retoile}
r^*_{\alpha,n,\rho}(r) = r\left(\frac{1}{2} + 0.1\ \text{C}_{\alpha, n, \rho}\right)\left(2 + \frac{1}{2}r + 0.1\ r\ \text{C}_{\alpha, n, \rho}\right)
\eea
where $\text{C}_{\alpha,n,\rho} = \sqrt{\alpha} + \sqrt{\frac{\log(n)}{\log(\rho^{-1})}}$.
Thus, using \ref{probacond} and Lemma \ref{E*alpha},

\bea 
\bP \left( 
\left\|X_{T^*}^t X_{T^*}-I \right\| \ge r^*_{\alpha,n,p}(r)  \right) & \le & 
\frac{218+1}{p^\alpha},
\eea 

\subsubsection{Control of $\displaystyle \max_{k\in {T^*}^c} \left\|X_{T^*}^t X_{k} \right\|_2$}
\label{C.2.2}
By the triangular inequality, we have that 
\begin{align} 
\nonumber & \max_{k\in {T^*}^c} \left\|X_{T^*}^t X_{k} \right\|_2 = 
\max_{k\in {T^*}^c} \left\|\left(\mathfrak{C}_{\mathcal K}+E_{T^*}\right)^t D_*^2
\left(\mathfrak{C}_{k}+E_{k} \right) \right\|_2 \\
& \hspace{3cm}\le  \Bigg( \max_{k\in {T^*}^c} \left\|\mathfrak{C}_{\mathcal K}^t \mathfrak{C}_{k}\right\|_2 
+ \left\|\mathfrak{C}_{\mathcal K}\right\| \max_{k\in {T^*}^c}  \left\|E_{k}\right\|_2 
\\
 & \hspace{3cm}
+ \left\|E_{T^*}\right\| \ \max_{k\in {T^*}^c}  \left\|\mathfrak{C}_{k} \right\|_2  + \left\|E_{T^*}\right\|\ \max_{k\in {T^*}^c}  \left\| E_{k}\right\|_2 \Bigg) \|D_*^2\|. \nonumber 
\end{align} 
On the other hand, we have
%A computation analogous to the one for the probability of $\mathcal E_\alpha$ gives that 
%\textcolor{red}{\bea 
%\bP\left( \max_{k\in \{1,\ldots,p\} }  \left\|E_{k}\right\|_2 \ge 
%\mathfrak{s} \left(\sqrt{n}+\sqrt{\frac{\alpha+1}{c} \log(p)}\right) \mid \mathcal E_\alpha^* \right) 
%& \le  & \frac{C}{p^\alpha}. 
%\eea }
\bea 
\bP\left( \max_{k\in \{1,\ldots,p\} }  \left\|E_{k}\right\|_2 \ge 
\mathfrak{s} \left(\sqrt{n}+\sqrt{\frac{\alpha+1}{c} \log(p)}\right) \mid \mathcal E_\alpha^* \right) 
& \le  & \frac{C}{p^{\alpha}}. 
\eea 
%Thus, we have
%\bean 
 %\max_{k\in {T^*}^c} \left\|X_{T^*}^t X_{k} \right\|_2 
%& \le \left(\frac{C_{col}\cdot \sqrt{\log(p)}}{1 - r}
%+\left(\sqrt{1+r}+\mathfrak{s} K_{n,s^*} \right) \mathfrak{s}
%\left(\sqrt{n}+\sqrt{\frac{\alpha+1}{c} \log(p)}\right) + \mathfrak{s} K_{n,s^*}\sqrt{\frac{C_{col}\cdot \sqrt{\log(p)}}{1 - r}}\right)\\
%& \times \frac{1}{\left(1-\mathfrak{s} \sqrt{n \left(\frac{\alpha(1-e^{-1})}
%{\vartheta_* \ C_\chi} \right)^{\frac1{n}} \left(\frac1{\log(p)^{\nu-1}}\right)^{\frac1{n}}}\right)^{2}}
%\eean
Thus, using Lemma \ref{Ccol} and Lemma \ref{normET*}, we obtain 
\begin{align}
 &   \nonumber \bP \Bigg(\max_{k\in {T^*}^c} \left\|X_{T^*}^t X_{k} \right\|_2 \ge  \Bigg(\frac{C_{col}\cdot \sqrt{\log(\rho^{-1})}}{1 - r}
+\Big(\frac{1}{\sqrt{1-r}}+\mathfrak{s} K_{n,s^*} \Big) \mathfrak{s}
\Big(\sqrt{n}+\sqrt{\frac{\alpha+1}{c} \log(p)}\Big) \\
 &\hspace{1cm}+ \mathfrak{s} K_{n,s^*}\frac{C_{col}\cdot \sqrt{\log(\rho^{-1})}}{\left(1 - r\right)^{\frac{3}{2}}}\Bigg) 
 \times \frac{1}{\left(1+\mathfrak{s} \sqrt{n \left(\frac{\alpha(1-e^{-1})}
{\vartheta_* \ C_\chi} \right)^{\frac1{n}} \left(\frac1{\log(p)^{\nu-1}}\right)^{\frac1{n}}}\right)^{2}}  \mid \mathcal E_\alpha^* \Bigg)\nonumber\\
& \hspace{1cm} \le \frac{C+2}{p^{\alpha}} + \frac{2}{\rho^{1  - \left(\alpha + 1\right) \log(K)^{2}}}.
%& \hspace{1cm} \le  \frac{C+2}{p^\alpha} + \frac{2}{p^{\left(\alpha + 1\right) \log(K)^{2} - 1}}.
\end{align}
Using the fact
\bean
\mathbb{P}\left(A\right) \leq \mathbb{P}\left(A \cap E\right) + \mathbb{P}\left(E^{c}\right)
\eean
with $A =\displaystyle \max_{k\in {T^*}^c} \left\|X_{T^*}^t X_{k} \right\|_2$ and 
\begin{align*}
    E = \Bigg(\displaystyle \max_{k\in {T^*}^c} \left\|\mathfrak{C}_{\mathcal K}^t \mathfrak{C}_{k}\right\|_2 
+ \left\|\mathfrak{C}_{\mathcal K}\right\| \displaystyle \max_{k\in {T^*}^c}  \left\|E_{k}\right\|_2 
+ \left\|E_{T^*}\right\| \ \displaystyle \max_{k\in {T^*}^c}  \left\|\mathfrak{C}_{k} \right\|_2  + \left\|E_{T^*}\right\|\ \displaystyle \max_{k\in {T^*}^c}  \left\| E_{k}\right\|_2 \Bigg) \|D_*^2\|.
\end{align*}
%\bea
%\mathbb P(A ) & = & \mathbb P(A \mid E \cup E^c) \\
%& = & \frac{\mathbb P(A \cap (E \cup E^c))}{\mathbb P(E \cup E^c)} \\
%& = & \mathbb P(A \cap (E \cup E^c)) \quad \text{car} \quad  \mathbb P(E \cup E^c) =1 \\
%& = & \mathbb P(A \cap E) +  \mathbb P(A \cap   E^c)) \\
%& \le & \mathbb P(A \cap E) +  \mathbb P(E^c)
%\eea
Furthermore, since  $A \subset E^c$ and $E=E_1\cap E_2 \cap E_3 \cap E_{4}$, we have by union bound,,
\bea
\mathbb P(A ) &  \le &  \mathbb P(E_1^c \cup E_2^c \cup E_3^c \cup E_{4}^{c}) \le \mathbb P(E_1^c) +\mathbb P(E_2^c) + \mathbb P(E_3^c) + \mathbb P(E_4^c).
\eea
%\bean
%\bP \left( \max_{k\in \mathcal K^c} 
%\left\|\mathfrak{C}_{\mathcal K}^t \mathfrak{C}_{k}\right\|_2\ge \frac{C_{col}\cdot \sqrt{\log(p)}}{1-r} \right) 
%& \le & 
%K_{o}\ \frac1{p^{\left(\alpha + 1\right) \log(K)^{2}}}
%\eean

%\bean 
%\bP\left( \max_{k\in \{1,\ldots,p\} }  \left\|E_{k}\right\|_2 \ge 
%\mathfrak{s} \left(\sqrt{n}+\sqrt{\frac{\alpha+1}{c} \log(p)}\right) \mid \mathcal E_\alpha^* \right) 
%& \le  & \frac{C}{p^\alpha}. 
%\eean 

%\bean 
%\bP\left(\left\|E_{T^*}^t\right\| 
%\ge \mathfrak{s} K_{n,s^*}  \mid \mathcal E^*_{\alpha}\right) & \le & \frac2{p^\alpha}
%\eean 

%\bea
%\left(\mid  \mathcal E_\alpha^*\right) &= & 
%\mathbb P\left(A \mid E_1 \cap E_2 \cap E_3 \cap \mathcal E_\alpha^*\right)
%\eea 
Since, by Assumption (\ref{ass3}),  
%\bean 
%\left(1+r+\mathfrak{s} K_{n,s^*} \right) \mathfrak{s}
%\left(\sqrt{n}+\sqrt{\frac{\alpha+1}{c} \log(p)}\right) 
%& \le & 
%(1+1.1\cdot r) \ \frac{C_{\mathfrak{s},n,p}}{\sqrt{\log(p)}},
%\eean
\bean 
\left(\frac{1}{\sqrt{1-r}}+\mathfrak{s} K_{n,s^*} \right) \mathfrak{s}
\left(\sqrt{n}+\sqrt{\frac{\alpha+1}{c} \log(p)}\right) 
& \le & 
\left(\frac{1}{\sqrt{1-r}} + 0.1\ r\ \text{C}_{\alpha, n, \rho}\right) \ \frac{C_{\mathfrak{s},n,p}}{\sqrt{\log(\rho^{-1})}},
\eean 
we obtain
\begin{align}
    & \nonumber \bP \Bigg( \max_{k\in {T^*}^c} \left\|X_{T^*}^t X_{k} \right\|_2  \ge  \Bigg(\frac{C_{col}\cdot \sqrt{\alpha \log(\rho^{-1})}}{1-r} + \left(\frac{1}{\sqrt{1-r}} + 0.1\ r\ \text{C}_{\alpha, n, \rho}\right)\frac{C_{\mathfrak{s},n,p}}{\sqrt{\log(\rho^{-1})}}\\ 
\nonumber & \hspace{1.7cm} + 0.1\ r\ \text{C}_{\alpha, n, p} \cdot\frac{C_{col}\cdot \sqrt{\alpha \log(\rho^{-1})}}{\left(1 - r\right)^{3/2}}\Bigg) \times \frac{1}{\left(1-\mathfrak{s} \sqrt{n \left(\frac{\alpha(1-e^{-1})}
{\vartheta_* \ C_\chi} \right)^{\frac1{n}} \left(\frac1{\log(p)^{\nu-1}}\right)^{\frac1{n}}}\right)^{2}}\mid \mathcal E_\alpha^* \Bigg)\\
%& \hspace{1.7cm} \le \left(C +2\right)\rho^{\alpha} + \frac{2}{\rho^{1  - \left(\alpha + 1\right) \log(K)^{2}}}.\\
 &\hspace{1.7cm}\le  \frac{C+2}{p^{\alpha}} + 2\rho^{\left(\alpha + 1\right) \log(K)^{2}- 1}.
\end{align}
%\frac{C_{col}+(1+1.1\cdot r) \ C_{\mathfrak{s},n,p}}{\sqrt{\log(p)}}
Moreover, for any event $\mathcal A$, 
\begin{align}
\label{probacond}
\bP\left( \mathcal A  \right)
 \le  \bP\left( \mathcal A \mid \mathcal E_{\alpha}\right)
+\bP\left(\mathcal E_{\alpha}^c\right),
\end{align}
%using \ref{probacond} 
and Lemma \ref{E*alpha}, we obtain 
\begin{align}
    \label{C2}
\nonumber\bP \Bigg( \max_{k\in {T^*}^c} \left\|X_{T^*}^t X_{k} \right\|_2 & \ge  \Bigg(\frac{C_{col}\cdot \sqrt{\alpha \log(\rho^{-1})}}{1-r} + \left(\frac{1}{\sqrt{1-r}} + 0.1\ r\ \text{C}_{\alpha, n, p}\right)\frac{C_{\mathfrak{s},n,p}}{\sqrt{\log(\rho^{-1})}}\\ 
\nonumber &  \hspace{-0.7cm} + 0.1\ r\ \text{C}_{\alpha, n, p} \cdot\frac{C_{col}\cdot \sqrt{\alpha \log(\rho^{-1})}}{\left(1 - r\right)^{3/2}}\Bigg) \times \frac{1}{\left(1-\mathfrak{s} \sqrt{n \left(\frac{\alpha(1-e^{-1})}
{\vartheta_* \ C_\chi} \right)^{\frac1{n}} \left(\frac1{\log(p)^{\nu-1}}\right)^{\frac1{n}}}\right)^{2}} \Bigg)\\
& \le  \frac{C+3}{p^{\alpha}} + 2\rho^{\left(\alpha + 1\right) \log(K)^{2}- 1}.
\end{align}

\subsection{The last two inequalities}
The proof of (\ref{orthcond}) is standard and, 
under Assumption \ref{ass3}, the proof of (\ref{compsize}) can be proved using the 
ideas of \cite[Section 3.3]{CandesPlan:AnnStat09}. We give the proofs for the 
sake of completeness. 
\subsubsection{Control of $\left\|X_{{T^*}^c}^tX_{T^*}(X_{T^*}^t X_{T^*})^{-1}X^t_{T^*}z\right\|_\infty$} 
For any $j\in {T^*}^c$, by the results of section \ref{C.2.2}, we have
%, and the fact that 
%\begin{align*}
    %\left\|X_{T^*}(X_{T^*}^t X_{T^*})^{-1}X^t_{T^*}X_{j}\right\|_2^2 & \le  \frac{1+r^*}{(1-r^*)^2}\frac{\left(C_{col}+(1+1.1\cdot r) \ C_{\mathfrak{s},n,p}
%\right)^2}{\log(p)}
%\end{align*}
%& \hspace{-2cm}
%\begin{align*}
 % & \left\|X_{T^*}(X_{T^*}^t X_{T^*})^{-1}X^t_{T^*}X_{j}\right\|_2  \le  \frac{1 + r^{*}}{\left(1 - r^{*}\right)\left(1-\mathfrak{s} \sqrt{n \left(\frac{\alpha(1-e^{-1})}{\vartheta_* \ C_\chi} \right)^{\frac1{n}} \Big(\frac1{\log(p)^{\nu-1}}\right)^{\frac1{n}}}\Big) 
  % \Bigg(\frac{0.1\ r}{\sqrt{n}}\ \text{C}_{\alpha, n, p} \cdot\frac{C_{col}\cdot \sqrt{\log(p)}}{1 - r}\\
   %&+   \frac{C_{col}\cdot \sqrt{\log(p)}}{1-r}  +\left(\frac{1}{\sqrt{1-r}} + \frac{0.1\ r}{\sqrt{n}}\text{C}_{\alpha, n, p}\right)\frac{C_{\mathfrak{s},n,p}}{\sqrt{\log(p)}} \Bigg)
%\end{align*}
\begin{align*}
    \Vert X_{T^{*}}(X_{T^*}^t X_{T^*})^{-1}X_{T^*}X_{j}^{t}\Vert_{2}  &\le  \frac{\sqrt{1 + r^*_{\alpha,n,p}(r)}}{\left(1 - r^*_{\alpha,n,p}(r)\right)\left(1-\mathfrak{s} \sqrt{n \left(\frac{\alpha(1-e^{-1})}{\vartheta_* \ C_\chi} \right)^{\frac1{n}} \left(\frac1{\log(p)^{\nu-1}}\right)^{\frac1{n}}}\right)^{2}}\\
& \hspace{-5cm}\times \Bigg(0.1\ r\ \text{C}_{\alpha, n, p} \cdot\frac{C_{col}\cdot \sqrt{\log(\rho^{-1})}}{\left(1 - r\right)^{3/2}}+ \frac{C_{col}\cdot \sqrt{\log(\rho^{-1})}}{1-r}  +\left(\frac{1}{\sqrt{1-r}} + 0.1\ r\ \text{C}_{\alpha, n, p}\right)\frac{C_{\mathfrak{s},n,p}}{\sqrt{\log(\rho^{-1})}} \Bigg)
\end{align*}

with probability at least $1-\left(\frac{C+3}{p^{\alpha}} + 2\rho^{\left(\alpha + 1\right) \log(K)^{2}- 1}\right) $, we get 
%\bean 
%\bP \left(X_{j}^tX_{T^*}(X_{T^*}^t X_{T^*})^{-1}X^t_{T^*}z\ge u\right) 
%& \le & \frac12\exp \left(-\frac{u^2}{2\sigma^2 \ 
%\frac{1+r^*}{(1-r^*)^2}\frac{\left(C_{col}+(1+1.1\cdot r) \ C_{\mathfrak{s},n,p}
%\right)^2}{\log(p)}}\right)+\frac{??}{p^\alpha}
%\eean 
\bean
&\bP \left(X_{j}^{t}X_{T^*}(X_{T^*}^t X_{T^*})^{-1}X^t_{T^*}z\ge u\right) 
 \le\\
& \frac{1}{2}\exp\left(-\frac{u^{2}\left(1-\mathfrak{s} \sqrt{n \left(\frac{\alpha(1-e^{-1})}
{\vartheta_{*} \ C_{\chi}} \right)^{\frac1{n}} \left(\frac1{\log(p)^{\nu-1}}\right)^{\frac1{n}}}\right)^{2}}{2\sigma^{2}\frac{\sqrt{1 + r^*_{\alpha,n,\rho}(r)}}{ 1 - r^*_{\alpha,n,\rho}(r)}\begin{pmatrix}\big(0.1 r\ \text{C}_{\alpha, n, p} \cdot\frac{C_{col}\cdot \sqrt{\log(\rho^{-1})}}{\left(1 - r\right)^{3/2}} +   \frac{C_{col}\cdot \sqrt{\log(\rho^{-1})}}{1-r} \\ +\Bigg(\frac{1}{\sqrt{1-r}} + 0.1 r\ \text{C}_{\alpha, n, \rho}\Big)\frac{C_{\mathfrak{s},n,p}}{\sqrt{\log(\rho^{-1})}}\big)
\end{pmatrix}}\right) 
%\\& \hspace{1cm}
+ \frac{C+3}{p^{\alpha}} +  \rho^{\left(\alpha + 1\right) \log(K)^{2}- 1}.
\eean
Taking $u$ such that 
%\bean 
%\frac12\exp \left(-\frac{u^2}{2\sigma^2 \ 
%\frac{1+r^*}{(1-r^*)^2}\frac{\left(C_{col}+(1+1.1\cdot r) \ C_{\mathfrak{s},n,p}
%\right)^2}{\log(p)}}\right) & = & \frac1{p^{\alpha}}
%\eean
\begin{align}
& \frac{1}{2}\exp\left(-\frac{u^{2}\left(1-\mathfrak{s} \sqrt{n \left(\frac{\alpha(1-e^{-1})}
{\vartheta_* \ C_\chi} \right)^{\frac1{n}} \left(\frac1{\log(p)^{\nu-1}}\right)^{\frac1{n}}}\right)^{2}}{2\sigma^{2}\frac{\sqrt{1 + r^*_{\alpha,n,\rho}(r)}}{ 1 - r^*_{\alpha,n,\rho}(r)}\begin{pmatrix}\Bigg(0.1 r\ \text{C}_{\alpha, n, p} \cdot\frac{C_{col}\cdot \sqrt{\log(\rho^{-1})}}{\left(1 - r\right)^{3/2}} +   \frac{C_{col}\cdot \sqrt{\log(\rho^{-1})}}{1-r} \\ +\Bigg(\frac{1}{\sqrt{1-r}} + 0.1 r\ \text{C}_{\alpha, n, \rho}\Big)\frac{C_{\mathfrak{s},n,p}}{\sqrt{\log(\rho^{-1})}}\Bigg)
\end{pmatrix}}\right) \nonumber = \rho^{\alpha}
%\\
%& \hspace{3cm}
\end{align}
i.e. 
%\bean 
%u  & = & \sqrt{\left(\alpha \log(p)-\log(2)\right) 2\sigma^2 \ \frac{1+r^*}{(1-r^*)^2}\frac{\left(C_{col}+(1+1.1\cdot r) \ C_{\mathfrak{s},n,p}
%\right)^2}{\log(p)}}.
%\eean
\begin{align}
    u & =  \nonumber \sqrt{\big(\alpha \log(\rho^{-1})-\log(2)\big)}\sqrt{\frac{2\sigma^{2}\frac{\sqrt{1 + r^*_{\alpha,n,\rho}(r)}}{ 1 - r^*_{\alpha,n,\rho}(r)}\begin{pmatrix}\Bigg(0.1 r\ \text{C}_{\alpha, n, p} \cdot\frac{C_{col}\cdot \sqrt{\log(\rho^{-1})}}{\left(1 - r\right)^{3/2}} +   \frac{C_{col}\cdot \sqrt{\log(\rho^{-1})}}{1-r} \\ +\Bigg(\frac{1}{\sqrt{1-r}} + 0.1 r\ \text{C}_{\alpha, n, \rho}\Big)\frac{C_{\mathfrak{s},n,p}}{\sqrt{\log(\rho^{-1})}}\Bigg)
\end{pmatrix}}{\Bigg(1-\mathfrak{s} \sqrt{n \Big(\frac{\alpha(1-e^{-1})}
{\vartheta_* \ C_\chi} \Big)^{\frac1{n}} \Big(\frac1{\log(p)^{\nu-1}}\Big)^{\frac1{n}}}\Bigg)^{2}}}
\end{align}
Using the union bound, we finally obtain 
%\bean 
%& & \bP \left(
%\left\|X_{{T^*}^c}^tX_{T^*}(X_{T^*}^t X_{T^*})^{-1}X^t_{T^*}z\right\|_\infty 
%\ge \sqrt{\left(\alpha \log(p)-\log(2)\right) 2\sigma^2 \ \frac{1+r^*}{(1-r^*)^2}\frac{\left(C_{col}+(1+1.1\cdot r) \ C_{\mathfrak{s},n,p}
%\right)^2}{\log(p)}}
%\right) \\
%& & \hspace{5cm} \le \frac{C+223}{p^{\alpha-1}}.
%\eean 
\begin{align}
    &\bP \begin{pmatrix}\left\|X_{{T^*}^c}^{t}X_{T^*}(X_{T^*}^t X_{T^*})^{-1}X^t_{T^*}z\right\|_\infty \ge \nonumber \\
\nonumber \\
\sqrt{\left(\alpha \log(\rho^{-1})-\log(2)\right)}\sqrt{\frac{2\sigma^{2}\ \frac{\sqrt{1 + r^*_{\alpha,n,\rho}(r)}}{1 - r^*_{\alpha,n,p}(\rho)}\begin{pmatrix}0.1\ r\ \text{C}_{\alpha, n, \rho} \cdot\frac{C_{col}\cdot \sqrt{\log(\rho^{-1})}}{\left(1 - r\right)^{3/2}} + \frac{C_{col}\cdot \sqrt{\log(\rho^{-1})}}{1-r} \\+\left(\frac{1}{\sqrt{1-r}} + 0.1\ r\ \text{C}_{\alpha, n, \rho}\right)\frac{C_{\mathfrak{s},n,p}}{\sqrt{\log(\rho^{-1})}}
\end{pmatrix}}{\left(1-\mathfrak{s} \sqrt{n \left(\frac{\alpha(1-e^{-1})}
{\vartheta_* \ C_\chi} \right)^{\frac1{n}} \left(\frac1{\log(p)^{\nu-1}}\right)^{\frac1{n}}}\right)^{2}}} \end{pmatrix}\\
\nonumber\\
& \le \frac{C+4}{p^{\alpha -1}} +  p \ \rho^{\left(\alpha + 1\right) \log(K)^{2} - 1}.
\end{align}

\subsubsection{Control of 
$\left\|X_{{T^*}^c}^tX_{T^*}(X_{T^*}^t X_{T^*})^{-1}\sgn\left(\beta^*_{T^*}\right)\right\|_\infty$}
For any $j\in {T^*}^c$, again by the results of section (C.2.2), we have
\bean
\Vert X_{j}^{t}X_{T^*}(X_{T^*}^t X_{T^*})^{-1}\Vert_{2}  &\le & \frac{\begin{pmatrix}\Bigg(0.1\ r\ \text{C}_{\alpha, n, \rho} \cdot\frac{C_{col}\cdot \sqrt{\log(\rho^{-1})}}{\left(1 - r\right)^{3/2}}\\
  +  \frac{C_{col}\cdot \sqrt{\log(p)}}{1-r}  +\left(\frac{1}{\sqrt{1-r}} + 0.1\ r\ \text{C}_{\alpha, n, \rho}\right)\frac{C_{\mathfrak{s},n,p}}{\sqrt{\log(\rho^{-1})}} \Bigg)
  \end{pmatrix}}{\left( 1 - r^*_{\alpha,n,\rho}(r)\right)\left(1-\mathfrak{s} \sqrt{n \left(\frac{\alpha(1-e^{-1})}{\vartheta_* \ C_\chi} \right)^{\frac1{n}} \left(\frac1{\log(p)^{\nu-1}}\right)^{\frac1{n}}}\right)^{2}}
\eean
with probability at least $1-\left(\frac{C+3}{p^{\alpha}} + 2\rho^{\left(\alpha + 1\right) \log(K)^{2}- 1}\right)$.
Hoeffding's inequality gives 
\begin{align}
    &\bP \left(X_{j}^{t}X_{T^*}(X_{T^*}^t X_{T^*})^{-1}\sgn\left(\beta^*_{T^*}\right)\ge u\right)  \le  
\frac{1}{2}\exp \left(-\frac{u^2}{2 \ 
\left\|(X_{T^*}^t X_{T^*})^{-1}X^t_{T^*}X_{j}\right\|_2}\right) \nonumber \\
%\frac12\exp \left(-\frac{u^2}{2 \ 
%\frac{\left(C_{col}+(1+1.1\cdot r) \ C_{\mathfrak{s},n,p}
%\right)^2}{\log(p)\ (1-r^*)^2}}\right)+\frac{C+219+3}{p^\alpha}.
& \le  \frac{1}{2}\exp\left(-\frac{u^{2}\left(1-\mathfrak{s} \sqrt{n \left(\frac{\alpha(1-e^{-1})}
{\vartheta_* \ C_\chi} \right)^{\frac1{n}} \left(\frac1{\log(p)^{\nu-1}}\right)^{\frac1{n}}}\right)^{2}}{\frac{2}{ 1 - r^*_{\alpha,n,\rho}(r)}\begin{pmatrix}\Bigg(0.1 r\ \text{C}_{\alpha, n, p} \cdot\frac{C_{col}\cdot \sqrt{\log(\rho^{-1})}}{\left(1 - r\right)^{3/2}} +   \frac{C_{col}\cdot \sqrt{\log(\rho^{-1})}}{1-r} \\ +\Bigg(\frac{1}{\sqrt{1-r}} + 0.1 r\ \text{C}_{\alpha, n, \rho}\Big)\frac{C_{\mathfrak{s},n,p}}{\sqrt{\log(\rho^{-1})}}\Bigg)
\end{pmatrix}}\right) \\
& \hspace{1cm}+ \frac{C+3}{p^{\alpha}} + 2\rho^{\left(\alpha + 1\right) \log(K)^{2}- 1}.\nonumber 
\end{align}
Choosing 
\begin{align}
u & = &  \sqrt{\frac{\frac{2\left(\alpha \log(\rho^{-1})-\log(2)\right)}{ 1 - r^*_{\alpha,n,\rho}(r)}
\begin{pmatrix}0.1\ r\ \text{C}_{\alpha, n, \rho} \cdot\frac{C_{col}\cdot \sqrt{\log(\rho^{-1})}}{\left(1 - r\right)^{3/2}} +   \frac{C_{col}\cdot \sqrt{\log(\rho^{-1})}}{1-r} \\ +\left(\frac{1}{\sqrt{1-r}} + 0.1\ r\ \text{C}_{\alpha, n, \rho}\right)\frac{C_{\mathfrak{s},n,p}}{\sqrt{\log(\rho^{-1})}}
\end{pmatrix}}
{\left(1-\mathfrak{s} \sqrt{n \left(\frac{\alpha(1-e^{-1})}
{\vartheta_* \ C_\chi} \right)^{\frac1{n}} \left(\frac1{\log(p)^{\nu-1}}\right)^{\frac1{n}}}\right)^{2}}}
\end{align}
%\bean 
%u  & = & \sqrt{\left(\alpha \log(p)-\log(2)\right) 2 \ \frac{\left(C_{col}+(1+1.1\cdot r) \ C_{\mathfrak{s},n,p}
%\right)^2}{\log(p)\ (1-r^*)^2}}.
%\eean 
and applying the union bound, we obtain 
\bean
&\bP \Bigg(\left\|X_{{T^*}^c}^tX_{T^*}(X_{T^*}^t X_{T^*})^{-1}\sgn\left(\beta^*_{T^*}\right)\right\|_\infty \\
&\ge \sqrt{\frac{\frac{2\left(\alpha \log(\rho^{-1})-\log(2)\right)}{1 - r^*_{\alpha,n,\rho}(r)}
\begin{pmatrix}0.1\ r\ \text{C}_{\alpha, n, \rho} \cdot\frac{C_{col}\cdot \sqrt{\log(\rho^{-1})}}{\left(1 - r\right)^{3/2}} +   \frac{C_{col}\cdot \sqrt{\log(\rho^{-1})}}{1-r} \\ +\left(\frac{1}{\sqrt{1-r}} + 0.1\ r\ \text{C}_{\alpha, n, \rho}\right)\frac{C_{\mathfrak{s},n,p}}{\sqrt{\log(\rho^{-1})}}\end{pmatrix}}{\left(1-\mathfrak{s} \sqrt{n \left(\frac{\alpha(1-e^{-1})}
{\vartheta_* \ C_\chi} \right)^{\frac1{n}} \left(\frac1{\log(p)^{\nu-1}}\right)^{\frac1{n}}}\right)^{2}}}\ \Bigg)\\
&\le  \frac{C+4}{p^{\alpha - 1}} +  p\ \rho^{\left(\alpha + 1\right) \log(K)^{2} - 1}.
\eean
%\bean 
%\bP \left(X_{j}^tX_{T^*}(X_{T^*}^t X_{T^*})^{-1}\sgn\left(\beta^*_{T^*}\right)\ge 
%\sqrt{\left(\alpha \log(p)-\log(2)\right) 2 \ \frac{\left(C_{col}+(1+1.1\cdot r) \ C_{\mathfrak{s},n,p}
%\right)^2}{\log(p)\ (1-r^*)^2}}\right) & \le & 
%\frac{C+223}{p^{\alpha-1}}.
%\eean 

\subsubsection{Summing up}

%Using Assumption \ref{asscol}, 
We obtain that 
\bean 
& & \left\|X_{{T^*}^c}^tX_{T^*}(X_{T^*}^t X_{T^*})^{-1}X^t_{T^*}z\right\|_\infty
+\lb \ \left\|X_{{T^*}^c}^tX_{T^*}(X_{T^*}^t X_{T^*})^{-1}\sgn\left(\beta^*_{T^*}\right)\right\|_\infty
\\
& & \hspace{2cm} \le 
%\textcolor{red}{\sigma \sqrt{1+1.1\cdot r\ \left(1.1 + 0.11 \cdot r\right)}+\frac12 \ \lb}
\sigma  \mathcal{C}_{1} + \lb \mathcal{C}_{2}
\eean
where

\begin{align}
\label{c1}
\mathcal{C}_{1} = \sqrt{\left(\alpha \log(\rho^{-1})-\log(2)\right)}\sqrt{\frac{\sqrt{1 + r^*_{\alpha,n,\rho}(r)}}{1 - r^*_{\alpha,n,\rho}(r)}\begin{pmatrix}0.1\ r\ \text{C}_{\alpha, n, \rho} \cdot\frac{C_{col}\cdot \sqrt{\log(\rho^{-1})}}{\left(1 - r\right)^{3/2}} +   \frac{C_{col}\cdot \sqrt{\log(\rho^{-1})}}{1-r} \\
\\
+\left(\frac{1}{\sqrt{1-r}} + \frac{0.1\ r}{\sqrt{n}}\text{C}_{\alpha, n, p}\right)\frac{C_{\mathfrak{s},n,p}}{\sqrt{\log(\rho^{-1})}}\end{pmatrix}}
\end{align}
and

\begin{align}
\label{c2}
\mathcal{C}_{2} =\sqrt{\frac{2\left(\alpha \log(\rho^{-1})-\log(2)\right)}{1 - r^*_{\alpha,n,\rho}(r)}\begin{pmatrix}
0.1\ r\ \text{C}_{\alpha, n, p} \cdot\frac{C_{col}\cdot \sqrt{\log(\rho^{-1})}}{\left(1 - r\right)^{3/2}} +   \frac{C_{col}\cdot \sqrt{\log(\rho^{-1})}}{1-r} \\
\\
+\left(\frac{1}{\sqrt{1-r}} + 0.1\ r\ \text{C}_{\alpha, n, p}\right)\frac{C_{\mathfrak{s},n,p}}{\sqrt{\log(\rho^{-1})}}
\end{pmatrix}}
\end{align}
as announced. Choosing $\rho$ such that Assumption \ref{pbound} is satisfied and $C_{col}$ is sufficiently small and $n$ sufficiently large such that $\mathcal C_1$ and $\mathcal C_2$ be smaller than 1/16 and the proof is completed.

\section{Conclusion}

The goal of this paper is to propose a sound study of the behavior of the LASSO algorithm for the linear model in the case where the design matrix is not satisfying the usual non-colinearity conditions that are enforced in standard analysis. We introduce a new model for the design matrix. In this new model, the columns are assumed to be drawn from a Gaussian mixture model where the centers of the mixture model, and them only, satisfy the incoherence condition. As a result, we are able to analyse an interesting example of applying the LASSO to a non-incoherent matrix and we obtain a performance bound. The price to pay for such a generality is that our prediction bounds hold with arbritrarily high but fixed probability, as compared with the incoherent setting where the  the probability goes to one as $p$ tends to $+\infty$.

\bibliographystyle{amsplain}
\bibliography{database}

\end{document}